\documentclass[reqno]{amsart}
\usepackage{amscd}
\usepackage[dvips]{graphics}
\usepackage{verbatim}
\makeatletter
    \newcommand{\rmnum}[1]{\romannumeral #1}
    \newcommand{\Rmnum}[1]{\expandafter\@slowromancap\romannumeral #1@}
\makeatother

\def\beq{\begin{equation}}
\def\eeq{\end{equation}}
\def\ba{\begin{array}}
\def\ea{\end{array}}

\def\R{\mathbb R}

\def\cal{\mathcal}

\newtheorem{thm}{Theorem}[section]
\newtheorem{lm}[thm]{Lemma}
\newtheorem{prop}[thm]{Proposition}
\newtheorem{crl}[thm]{Corollary}

\theoremstyle{definition}
\newtheorem{rem}[thm]{Remark}

\theoremstyle{remark}
 
\numberwithin{equation}{section}

\begin{document}
\pagestyle{plain}

\title{Subcritical Approach to Sharp Hardy-Littlewood-Sobolev Type Inequalities on the upper half space}
\author
 {Jingbo Dou, Qianqiao Guo and Meijun Zhu}
\address{Jingbo Dou, School of Statistics, Xi'an University of Finance and
Economics,Xi'an, Shaanxi, 710100, China}

\address{Qianqiao Guo, Department of Applied Mathematics, Northwestern Polytechnical University, Xi'an, Shaanxi, 710129, China}

\address{ Meijun Zhu, Department of Mathematics,
The University of Oklahoma, Norman, OK 73019, USA}

\address{}

 \maketitle
\begin{abstract}
In this paper we establish  the reversed sharp Hardy-Littlewood-Sobolev (HLS for short) inequality on the upper half space and obtain a new  HLS type integral inequality  on the upper half space (extending an inequality found by Hang, Wang and Yan in \cite{HWY2008})  by introducing a uniform approach. The extremal functions are classified via the method of  moving spheres, and the best constants are computed. The new approach can also be applied to obtain the classical HLS inequality and other similar inequalities.

\end{abstract}

\section{Introduction}

In this paper, we shall continue our study of the  extension of the sharp Hardy-Littlewood-Sobolev inequality.
The classical sharp Hardy-Littlewood-Sobolev (HLS) inequality (\cite{HL1928, HL1930, So1963, Lieb1983}) states that
\begin{equation}\label{class-HLS}
|\int_{\Bbb{R}^n} \int_{\Bbb{R}^n} f(x)|x-y|^{-(n-\alpha)} g(y) dx dy|\le N(p,\alpha, n)||f||_p||g||_t
\end{equation}
for all $f\in L^p(\Bbb{R}^n), \, g\in L^t(\Bbb{R}^n),$  where $1<p, \, t<\infty$, $0<\alpha <n$ and
$
1/p +1 /t+ (n-\alpha)/n=2.$
Lieb \cite{Lieb1983} proved the existence of the extremal functions to the inequality with the sharp constant, classified the extremal functions and computed the best constant in the case of $t=p$ (the conformal case), or $p=2$, or $t=2$ (that can be reduced to conformal case).   The sharp HLS inequality will be called the Hardy-Littlewood-Sobolev-Lieb (HLSL) inequality.

HLSL inequality plays a crucial role in the development of global  analysis. For instance, HLSL inequality implies Moser-Trudinger-Onofri and Beckner inequalities \cite{B1993}, which are widely used in the study of uniformization theorem and  prescribing curvature problems;  HLSL inequality implies sharp Sobolev inequality, as well as Gross's logarithmic Sobolev inequality \cite{Gr} (see also Weissler \cite{We1978}), which are widely used in conformal geometry and  Ricci flow problem.

Recently, we are able to extend HLSL inequalities in two directions: the sharp inequality on the upper half space  and the reversed HLS inequality.

On the upper half space  $\R^n_+=\{x\in\R^n|x=(x_1,x_2,\cdots,x_n),x_n>0 \}$. Dou and Zhu \cite{DZ-2} showed that: for $0<\alpha<n$, $1<p,t<\infty, \frac{n-1}{n}\cdot\frac{1}{p}+\frac{1}{t}+\frac{n-\alpha+1}{n}=2$,  there is a constant $C(p,n,\alpha)$, such that,  for any  $f\in L^p(\partial \R^n_+), g\in L^t(\R^n_+)$,
\begin{equation}\label{ineq-HLS-DZ}
|\int_{\R^n_+}\int_{\partial \R^n_+}\frac{g(x)f(y)}{|x-y|^{n-\alpha}} dy dx| \le C(p,n,\alpha)\|f\|_{L^p(\partial \R^n_+)}\|g\|_{L^t(\R^n_+)}.
\end{equation}
Moreover, the extremal functions are classified and the best constant is computed in conformal case, that is,  $p=\frac{2(n-1)}{n+\alpha-2}$ and $t=\frac{2n}{n+\alpha}$.

On the other hand, if $\alpha>n$, the reversed HLS inequality was found  by Dou and Zhu in  \cite{DZ-1}: for $\alpha>n$, $0<p,t<1, \frac{1}{p}+\frac{1}{t}+\frac{n-\alpha}{n}=2$, there is a constant  $C(p,n,\alpha)>0$, such that, for nonnegative functions $f\in L^p(\R^n), g\in L^t(\R^n)$,
\begin{equation}\label{ineq-RHLS}
\int_{\R^n}\int_{\R^n}\frac{g(x)f(y)}{|x-y|^{n-\alpha}} dy dx \ge C(p,n,\alpha)\|f\|_{L^p(\R^n)}\|g\|_{L^t(\R^n)}.
\end{equation}
 Again, in conformal case, that is,  $t=p=\frac{2n}{n+\alpha}$, the extremal functions are also classified and the best constant is computed in \cite{DZ-1}.

In this paper, we shall introduce a uniform approach to deal with these inequalities in conformal cases.  Note in conformal cases, the corresponding extremal functions can always  be classified and  the best constant can be explicitly computed.

Throughout the paper, we always use $B^n$ to represent the ball in $\R^n$ with radius 1 and centered at $(0, ..., 0, -1)$. The main idea of this approach is to study the corresponding sharp inequalities with subcritical exponents either on the  ball  $B^n$ or on the sphere $\partial B^n$, and show that the extremal functions in the cases are always constant functions, thus the best constants can be easily computed.  We then pass to the limit to obtain the expected sharp inequalities with critical exponents.

We first use this approach to establish the sharp reversed HLS inequality on the upper half space in conformal case.
\begin{thm}\label{HLS-UHS}
Assume that $\alpha>n$, $p=\frac{2(n-1)}{n+\alpha-2}, t=\frac{2n}{n+\alpha}$. Then for all  nonnegative $f\in L^p(\partial \R^n_+), g\in L^t(\R^n_+):$
\begin{equation}\label{ineq-HLS-UHS}
\int_{\R^n_+}\int_{\partial \R^n_+}\frac{g(x)f(y)}{|x-y|^{n-\alpha}} dy dx\ge C_{e_1}(n,\alpha)\|f\|_{L^p(\partial \R^n_+)}\|g\|_{L^t(\R^n_+)},
\end{equation}
where $$C_{e_1}(n,\alpha)=(n\omega_n)^{-\frac{n+\alpha-2}{2(n-1)}}\big(\int_{B^n}\big(\int_{\partial B^n}\frac{1}{|\xi-\zeta|^{n-\alpha}}d\zeta\big)^{\frac{2n}{n-\alpha}}d\xi\big)^{\frac{n-\alpha}{2n}}.$$
And the equality holds if and only if
$$f(y)=c \big(\frac 1{|y-\overline{y}^0|^2+d^2}\big)^{\frac{n+\alpha-2}2},
$$
for some $c>0, \  \overline{y}^0 \in \partial \R^n_+.$
\end{thm}
We indicated in \cite{DZ-1} that the above inequality holds. Late it was also proved by Ng\^{o} and  Nguyen \cite {NN-1} via  a different method, while they classified the extremal functions along the line of early work of Li \cite{Li} and Dou and Zhu \cite{DZ-1}.

We shall also  use the approach to establish  following new sharp  HLS type integral inequality.
\begin{thm}\label{HLST-UHS}
Assume $n\ge3, 2\le \alpha<n$, $p=\frac{2(n-1)}{n+\alpha-4}, t=\frac{2n}{n+\alpha}$, and write $x=(x^\prime,x_n)\in\R^n_+$.   Then the following sharp inequality holds for all nonnegative $f\in L^p(\partial \R^n_+), g\in L^t(\R^n_+):$
\begin{equation}\label{ineq-HLST-UHS}
\int_{\R^n_+}\int_{\partial \R^n_+}\frac{x_n g(x)f(y)}{(|x^\prime-y|^2+x_n^2)^{\frac{n-\alpha+2}{2}}} dy dx\le \frac{C_{e_2}(n,\alpha)}{2}\|f\|_{L^p(\partial \R^n_+)}\|g\|_{L^t(\R^n_+)},
\end{equation}
where
\begin{equation}\label{9-19-2}
C_{e_2}(n,\alpha)=(n\omega_n)^{-\frac{n+\alpha-4}{2(n-1)}}\big(\int_{B^n}\big(\int_{\partial B^n}\frac{(1-|\xi-z^o|^2)}{|\xi-\zeta|^{n-\alpha+2}}d\zeta\big)^{\frac{2n}{n-\alpha}}d\xi\big)^{\frac{n-\alpha}{2n}},
\end{equation}
and $z^o=(0,-1)\in \R^{n-1}\times\R$. And the equality holds if and only if
$$f(y)=c \big(\frac 1{|y-\overline{y}^0|^2+d^2}\big)^{\frac{n+\alpha-4}2},
$$
for some $c>0, \  \overline{y}^0 \in \partial \R^n_+.$
\end{thm}


Let
\begin{equation}\label{K-1}
K_{1, \alpha} (x, y)=\frac 1{|x-y|^{n-\alpha}}.
\end{equation}
HLS type inequalities essentially are the $L^p$ estimates for the convolution operators with this kernel,
either in the whole space (the classical HLS inequality), or on the upper half space (Dou and Zhu's generalization of HLS inequality on the upper half space \cite{DZ-2}).  Note that this kernel, up to a constant, can be viewed as the fundamental solution of $(-\Delta)^{\alpha/2} $ operator.

On the other hand,  consider
\begin{equation}\label{K-1}
K_{2, \alpha} (x, y):=-\frac1 {n-\alpha}\cdot \frac{\partial}{\partial x_n}K_{1. \alpha}=\frac {x_n}{|x-y|^{n+2-\alpha}}
\end{equation}
for $x\in \R^n_+,y\in \partial \R^n_+, 2\le\alpha<n$. This function, up to a constant, can be viewed as the fundamental solution of $(-\Delta)^{\alpha/2} $ operator on the upper half space. In fact, for $\alpha=2$, this is the classical Poisson kernel (up to a certain constant multiplier). From this view point, it is quite natural to expect that inequality \eqref{ineq-HLST-UHS}  shall hold.

In fact, for  $\alpha=2$, \eqref{ineq-HLST-UHS} was first obtained by Hang, Wang and Yan \cite{HWY2008}; and it  can be seen as the higher dimensional version of Carleman's inequality and is closely related to the sharp isoperimetric inequality. Hang, Wang and Yan's inequality  was also extended by Chen \cite{Chen2014} with different kernel function (for viewing fractional Lapalacian operator on the boundary of the upper half space as a special map from Dirichlet data to Neumann data). Chen's approach is along the same line as that in \cite{HWY2008}.


It is worth pointing out that the approach we introduced here can be used to give another proof  of  the classical HLSL inequality,  as well as that of reversed sharp HLS inequality in conformal case. In these conformal cases, there always are equivalent sharp inequalities either on unit sphere/ball, or on the upper half unit sphere/ball. These sharp inequalities usually play essential role in the study of geometric elliptic equations or integral equations involving critical exponents, such as  the sharp Sobolev inequality in study of  Yamabe problem, the sharp logarithmic Sobolev inequality in the study of Ricci flow.

We arrange the paper as follows. In Section \ref{Section 2},  we prove  Theorem \ref{HLS-UHS}. In fact, we prove its  equivalent form on a ball (Theorem \ref{HLS-ball} below). A reversed Young inequality is derived, which yields the inequality with  subcritical power (that is: $0<p<\frac{2(n-1)}{n+\alpha-2}, 0<t<\frac{2n}{n+\alpha}$). We then establish the existence of extremal functions under an extra  assumption on the boundedness of  $L^2$ norm. Then we classify the extremal functions via the method of moving planes, and compute the best constant.  Passing to the limit, we obtain  the sharp inequality for $p=\frac{2(n-1)}{n+\alpha-2}$ and $ t=\frac{2n}{n+\alpha}$. The extremal functions for the sharp inequality in critical case can be classified via the method of moving spheres (following from similar argument in Li \cite{Li} and Dou and Zhu \cite{DZ-2}). In Section \ref{Section 3}, we prove Theorem \ref{HLST-UHS} by similar procedures as in Section \ref{Section 2}.

Notation: throughout  the whole paper, we will write  $x^o=(0,-2), z^o=(0,-1)\in \R^{n-1}\times\R$ and $ B^n:=B_1(z^o)$ being the unit ball centered at $z^o$.


{\bf Acknowledgment}
This work is partially supported by the National Natural Science Foundation of China (Grant No.11571268). Q. Guo is also partially supported by the Fundamental Research Funds for the Central Universities (Grant No. 3102015ZY069) and the Natural Science Basic Research Plan in Shaanxi Province of China (Grant No. 2016JM1008). Zhu is partially supported by Simons foundation;  Q. Guo also would like to thank Department of Mathematics in  the  University of Oklahoma  for its hospitality,  where part of this work has been done.

\section{Reversed Sharp HLS Inequality on the Upper Half Space}\label{Section 2}
In this section we shall establish the  sharp reversed HLS inequality on the upper half space.

Consider  conformal  transformation $T_{x^o}:   x\in B^n  \to x^{x^o} \in \R^n_+$  given by
\begin{equation}\label{confor-map}
x^{x^o}=\frac{2^2(x-x^o)}{|x-x^o|^2}+x^o \in  \R^n_+.
\end{equation}
Via this conformal transformation,  one can easily show  that  inequality \eqref{ineq-HLS-UHS} is equivalent to the following sharp reversed HLS inequality on  the ball $B^n$ (see, e.g. Dou and Zhu \cite{DZ-2}):
\begin{thm}\label{HLS-ball}
For $\alpha>n$, $p=\frac{2(n-1)}{n+\alpha-2}, t=\frac{2n}{n+\alpha}$, the following sharp inequality holds for all  nonnegative $f\in L^p(\partial B^n), g\in L^t(B^n):$
\begin{equation}\label{ineq-HLS-ball}
\int_{B^n}\int_{\partial B^n}\frac{g(x)f(y)}{|x-y|^{n-\alpha}} dy dx\ge C_{e_1}(n,\alpha)\|f\|_{L^p(\partial B^n)}\|g\|_{L^t(B^n)},
\end{equation}
where $C_{e_1}(n,\alpha)$ is the sharp constant appearing in Theorem \ref{HLS-UHS}.
\end{thm}


\subsection{A Reversed Young Inequality}

We first establish a reversed Young inequality on a ball, which yields a reversed HLS inequality with subcritical powers.

For $x\in B^n\setminus \{z^o\}$, denote $x^*=z^o+\frac{x-z^o}{|x-z^o|}\in \partial B^n$ and $(z^o)^*=0$.



\begin{lm}\label{RY-ball}
Let $p,t\in (0,1)$ and $\frac{1}{p}+\frac{1}{p^\prime}=1, \frac{1}{t}+\frac{1}{t^\prime}=1$. Let $h(x, y)$ and $h_1(x,y)$ be two nonnegative functions defined on $ B^n \times \partial B^n$ and $\partial B^n \times \partial B^n$ respectively, satisfying  $h(x,y)\ge h_1(x^*,y), \forall~ x\in B^n, y\in \partial B^n$, $ h(\cdot, x^o)\in L^{at^\prime}(B^n), h_1(0, \cdot)\in L^{(1-a)p^\prime}(\partial B^n)$ for some $a\in (0, 1)$, and  $\|h(\cdot,y^1)\|_{L^{at^\prime}( B^n)}=\|h(\cdot,y^2)\|_{L^{at^\prime}( B^n)}, ~\forall~ y^1,y^2\in \partial B^n$, $\|h_1(z^1,\cdot)\|_{L^{(1-a)p^\prime}(\partial B^n)}=\|h_1(z^2,\cdot)\|_{L^{(1-a)p^\prime}(\partial B^n)}, ~\forall~ z^1,z^2\in \partial B^n$. Then  for all nonnegative functions $f\in L^p(\partial B^n), g\in L^t(B^n),$
\begin{equation}\label{equ-RY-ball}
\int_{B^n}\int_{\partial B^n}g(x)f(y)h(x,y)dy dx\ge\|f\|_{L^p(\partial B^n)}\|g\|_{L^t(B^n)}\|h(\cdot, x^o)\|_{L^{at^\prime}(B^n)}\|h_1(0,\cdot)\|_{L^{(1-a)p^\prime}(\partial B^n)}.
\end{equation}
\end{lm}
\begin{proof}
Let $q<0$ such that $\frac{1}{p}+\frac{1}{t}+\frac{1}{q}=2$ and $\frac{1}{q}+\frac{1}{q^\prime}=1.$ Denote
$$I=\int_{B^n}\int_{\partial B^n}g(x)f(y)h(x,y)dy dx$$ and
\begin{eqnarray*}
\gamma_1(x,y)&=&g^{\frac{t}{p^\prime}}(x)h^{1-a}(x,y),\\
\gamma_2(x,y)&=&f^{\frac{p}{t^\prime}}(y)h^a(x,y),\\
\gamma_3(x,y)&=&g^{\frac{t}{q^\prime}}(x)f^{\frac{p}{q^\prime}}(y).
\end{eqnarray*}
By reversed H\"{o}lder inequality and H\"{o}lder inequality, we have
\begin{eqnarray*}
I&\ge&(\int_{B^n}\int_{\partial B^n}\gamma_3^{q^\prime}(x,y)dy dx)^{\frac{1}{q^\prime}}(\int_{B^n}\int_{\partial B^n}[\gamma_1(x,y)\gamma_2(x,y)]^q dy dx)^{\frac{1}{q}}\\
&\ge&(\int_{B^n}\int_{\partial B^n}\gamma_3^{q^\prime}(x,y)dy dx)^{\frac{1}{q^\prime}}(\int_{B^n}\int_{\partial B^n}\gamma_1^{p^\prime}(x,y) dy dx)^{\frac{1}{p^\prime}}(\int_{B^n}\int_{\partial B^n}\gamma_2^{t^\prime}(x,y) dy dx)^{\frac{1}{t^\prime}}\\
&=&\|\gamma_3\|_{L^{q^\prime}(B^n\times\partial B^n)}\|\gamma_1\|_{L^{p^\prime}(B^n\times\partial B^n)}\|\gamma_2\|_{L^{t^\prime}(B^n\times\partial B^n)}.
\end{eqnarray*}
On the other hand, it is easy to see that
\begin{equation*}
\|\gamma_3\|_{L^{q^\prime}(B^n\times\partial B^n)}=(\int_{B^n}\int_{\partial B^n}g^t(x)f^p(y)dy dx)^{\frac{1}{q^\prime}}=\|g\|^{\frac{t}{q^\prime}}_{L^t(B^n)}\|f\|^{\frac{p}{q^\prime}}_{L^p(\partial B^n)},
\end{equation*}
and
\begin{eqnarray*}
\|\gamma_2\|_{L^{t^\prime}(B^n\times\partial B^n)}&=&(\int_{B^n}\int_{\partial B^n}f^p(y)h^{at^\prime}(x,y)dy dx)^{\frac{1}{t^\prime}}\\
&=&(\int_{\partial B^n} f^p(y)dy)^{\frac{1}{t^\prime}}(\int_{B^n} h^{at^\prime}(x,x^o)dx)^{\frac{1}{t^\prime}}\\
&=&\|f\|^{\frac{p}{t^\prime}}_{L^p(\partial B^n)}\|h(\cdot,x^o)\|^a_{L^{at^\prime}(B^n)}.
\end{eqnarray*}

Since $h(x,y)\ge h_1(x^*,y)\ge0, \forall~ x\in B^n, y\in \partial B^n$  and $\|h_1(z^1,\cdot)\|_{L^{(1-a)p^\prime}(\partial B^n)}=\|h_1(z^2,\cdot)\|_{L^{(1-a)p^\prime}(\partial B^n)}, ~\forall~ z^1,z^2\in \partial B^n$, we have
\begin{eqnarray*}
\|\gamma_1\|_{L^{p^\prime}(B^n\times\partial B^n)}&=&(\int_{B^n}\int_{\partial B^n}g^t(x)h^{(1-a)p^\prime}(x,y)dy dx)^{\frac{1}{p^\prime}}\\
&\ge&(\int_{B^n}\int_{\partial B^n}g^t(x)h_1^{(1-a)p^\prime}(x^*,y)dy dx)^{\frac{1}{p^\prime}}\\
&=&(\int_{B^n} g^t(x)dx)^{\frac{1}{p^\prime}}(\int_{\partial B^n} h_1^{(1-a)p^\prime}(0,y)dy)^{\frac{1}{p^\prime}}\\
&=&\|g\|^{\frac{t}{p^\prime}}_{L^t( B^n)}\|h_1(0, \cdot)\|^{1-a}_{L^{(1-a)p^\prime}(\partial B^n)},
\end{eqnarray*}
where $\int_{B^n}\int_{\partial B^n}g^t(x)h^{(1-a)p^\prime}(x,y)dy dx$ and $\int_{B^n}\int_{\partial B^n}g^t(x)h_1^{(1-a)p^\prime}(x^*,y)dy dx$ could  be $0$ and $+\infty$, in which case the above inequality holds automatically.

Lemma follows from above estimates on $\gamma_i$ for $i=1,2,3$.
\end{proof}

\begin{crl}\label{RYW-ball}
Assume that $\alpha>n,0<p<\frac{2(n-1)}{n+\alpha-2}, 0<t<\frac{2n}{n+\alpha}$. There is a positive constant $C(n,\alpha,p,t)>0$, such that: for all nonnegative $f\in L^p(\partial B^n), g\in L^t(B^n)$, there holds
\begin{equation}\label{IYW-ball}
\int_{B^n}\int_{\partial B^n}\frac{g(x)f(y)}{|x-y|^{n-\alpha}}dy dx\ge C(n,\alpha,p,t)\|f\|_{L^p(\partial B^n)}\|g\|_{L^t(B^n)}.
\end{equation}
\end{crl}
\begin{proof} In Lemma \ref{RY-ball}, we choose
$$h(x,y)=\frac{1}{|x-y|^{n-\alpha}}$$
and
 $$h_1(z,y)=\begin{cases}
1, \quad& \text{if}~|z-y|>\sqrt{2},\\
\frac{1}{(\frac{|z-y|}{2})^{n-\alpha}},\quad& \text{if}~|z-y|\le\sqrt{2}.
\end{cases}$$
where $x\in B^n, y,z \in\partial B^n$.

It is easy to see  that
$ h(\cdot, x^o)\in L^{\frac{t^\prime}{2}}(B^n), h_1(0, \cdot)\in L^{\frac{p^\prime}{2}}(\partial B^n)$, since
$$(n-\alpha)\frac{t^\prime}{2}<n, \ \ \ \ \mbox{and}  \  \  \   (n-\alpha)\frac{p^\prime}{2}<n-1,$$
which follow from the assumption that  both $p, \ t$ are subcritical, that is, $0<p<\frac{2(n-1)}{n+\alpha-2}, \ 0<t<\frac{2n}{n+\alpha}$.

It is also easy to verify $\|h_1(z^1,\cdot)\|_{L^{\frac{p^\prime}{2}}(\partial B^n)}=\|h_1(z^2,\cdot)\|_{L^{\frac{p^\prime}{2}}(\partial B^n)}, ~\forall~ z^1,z^2\in \partial B^n$; as well as  $h(x,y)\ge h_1(x^*,y), ~\forall~ x\in B^n, y\in \partial B^n$. Corollary \ref{RYW-ball} then follows from  Lemma \ref{RY-ball}.
\end{proof}

\begin{rem}
From above proof it is easy to see that inequality \eqref{IYW-ball} still holds if  $0<p\le\frac{2(n-1)}{n+\alpha-2}, 0<t<\frac{2n}{n+\alpha}$ or $0<p<\frac{2(n-1)}{n+\alpha-2}, 0<t\le\frac{2n}{n+\alpha}$ just by taking a different $a\in(0,1)$. More generally, the inequality \eqref{IYW-ball} holds for any $p,t\in(0,1)$ satisfying
$$1-\frac{n-1}{n-\alpha}(1-\frac 1p) <\frac n{n-\alpha}(1-\frac 1t).
$$
\end{rem}

\subsection{Sharp inequality in subcritical case} For $f(y)$ defined on $\partial B^n$, we introduce the extension operator as in \cite{DZ-2}:
$$E_\alpha f(x)=\int_{\partial B^n}\frac{f(y)}{|x-y|^{n-\alpha}}dy,~\text{for all}~x\in B^n.$$
Then  inequality \eqref{IYW-ball} is equivalent to
\begin{equation}\label{IYW-ball-equv}
\|E_\alpha f\|_{L^{t^\prime}(B^n)}\ge C(n,\alpha,p,t)\|f\|_{L^p(\partial B^n)},
\end{equation}
for nonnegative function $f(y)$.
To obtain the best constant $C(n,\alpha,p,t)$,
we consider
\begin{equation} \label{inf-UHS-sub}
\xi_\alpha:=\inf\{\|E_\alpha f\|_{L^{t^\prime}(B^n)}\ \ : \ \ f \ge 0, \ \ \ \|f\|_{L^p(\partial B^n)}=1, \  \|f\|_{L^2(\partial B^n)}\le b_2 \}
\end{equation}
for any $b_2 \ge \omega_{n-1}^{1/p-1/2}$, where $\omega_{n-1} $ is the surface area of $n-1$-dimension sphere $\partial B^n$. We first show that the above infimum can be achieved for subcriticall exponents $p, \ t$.


\begin{prop}\label{thm-exist-UHS-sub}
Assume that $\alpha>n,0<p\le \frac{2(n-1)}{n+\alpha-2}, 0<t\le \frac{2n}{n+\alpha}$. Then the infimum in \eqref{inf-UHS-sub} is attained by a nonnegative $f^*\in L^p(\partial B^n)$ with $\|f^*\|_{L^p(\partial B^n)}=1$.\
\end{prop}
\begin{proof}
%

 Let $\{f_i\}$ be a nonnegative minimizing sequence of \eqref{inf-UHS-sub} with $\|f_i\|_{L^p(\partial B^n)}=1$, and $\|f_i\|_{L^2(\partial B^n)}\le b_2$. Then, up to a subsequence, $f_i^p \to (f^*)^p$ weakly in $L^{2/p}(\partial B^n)$. Thus, $f^*\ge 0$, $\int_{\partial B^n}(f^*)^p=1$.
Since
$$ \int_{\partial B^n} f_i^p \cdot (f^*)^{2-p} \to \int_{\partial B^n}  (f^*)^{2},$$
we know via H\"{o}lder inequality that $\|f^*\|_{L^2} \le\lim_{i\to \infty} \|f_i\|_{L^2} \le b.$

On the other hand, noticing that up to a subsequence, $f_i^p \to (f^*)^p$ weakly in $L^{1/p}(\partial B^n)$, for any $x \in B^n$, $(f^*)^{1-p}(y)|x-y|^{\alpha-n} \in L^{\frac{1}{1-p}}(\partial B^n)$.
Thus
$$ \int_{\partial B^n} f_i^p \cdot (f^*)^{1-p}|x-y|^{\alpha-n} dy  \to \int_{\partial B^n}  f^*|x-y|^{\alpha-n}dy.$$
Using H\"{o}lder inequality again, we have
$$
\underline {\lim}_{i \to \infty} E_\alpha f_i(x) \ge E_\alpha f^*(x).
$$
Note  $\|f_i\|_{L^2(\partial B^n)}\le b_2$. Then, up to a subsequence, $f_i \to F$ weakly in $L^{1}(\partial B^n)$. Thus $E_\alpha f_i (x) \to E_\alpha F(x) $ pointwise. In fact, since $f_i$ is uniformly bounded in $L^1(\partial B^n),$ we know that $\{E_\alpha f_i(x)\}_{i=1}^\infty$ is uniformly bounded and equicontinuous on $\overline{B^n}$. Thus, up to a subsequence, $E_\alpha f_i (x) \to E_\alpha F(x)$ uniformly  in $B^n$. If $\|F\|_{L^1(\partial B^n)} =0$, then $(E_\alpha f_i)^{t'} \to \infty$, thus $\|E_\alpha f_i\|_{L^{t'}} \to 0$, which is a contradiction to the fact that  $\|E_\alpha f_i\|_{L^{t'}} \to \xi_\alpha >0$. So, $\|F\|_{L^1(\partial B^n)} >0$, thus $E_\alpha F >0$. It follows that  $(E_\alpha f_i (x))^{t'}  \to (E_\alpha F(x))^{t'}  $ uniformly  in $B^n$. Therefore,
$$
{\lim}_{i \to \infty} \int_{B^n}(E_\alpha f_i(x))^{t'}dx =  \int_{B^n} {\lim}_{i \to \infty} (E_\alpha f_i(x))^{t'}dx \le \int_{B^n} (E_\alpha f^*(x))^{t'} dx,
$$
which yields $\|E_\alpha f^*\|_{L^{t'}} \le \xi_\alpha$, that is:  $f^*$ is a non-negative minimizer.


\end{proof}


\subsection{Best constant for the reversed inequality with subcritical exponents}

%
%

Standard variational argument shows that $f^*(x)>0$. Also, near any point $z \in \partial B^n$, if $\phi(x) $ is in $L^2(\partial B^n \cap B_\epsilon (z))$, then, up to a constant multiplier,
$$
\int_{\partial B^n}f^{p-1}(z)\phi(z) dz=\int_{\partial B^n}\int_{B^n}\frac{(E_\alpha f(\xi))^{t^\prime-1}}{|\xi-z|^{n-\alpha}}\phi(z)d\xi dz.
$$
Since $E_\alpha f^*$ is continuous, we know that $f^*$
satisfies
\begin{equation}\label{equ-sub}
f^{p-1}(z)=\int_{B^n}\frac{(E_\alpha f(\xi))^{t^\prime-1}}{|\xi-z|^{n-\alpha}}d\xi, \ \ ~\forall~z\in\partial B^n.
\end{equation}
We will show that for $0<p< 2(n-1)/(n+\alpha-2)$ and $ 0<t < 2n/(n+\alpha),$  solution $f(z)=constant$ for any $z \in \partial B^n$ via moving plane method, thus to conclude that
\begin{equation}\label{9-19-1}
\xi_\alpha=(n\omega_n)^{-\frac{1}{p}}\big(\int_{B^n}\big(\int_{\partial B^n}\frac{1}{|\xi-\zeta|^{n-\alpha}}d\zeta\big)^{t^\prime}d\xi\big)^{\frac{1}{t^\prime}}.
\end{equation}



Let $y=z^{x^o}\in \partial \R^n_+, x=\xi^{x^o}\in \R^n_+$, where $z\in \partial B^n, \xi\in B^n$. Then we have
$$f^{p-1}(y^{x^o})=\int_{\R^n_+}\frac{(E_\alpha f(x^{x^o}))^{t^\prime-1}}{|x^{x^o}-y^{x^o}|^{n-\alpha}}\big(\frac{2}{|x-x^o|}\big)^{2n}dx. $$
Notice that
$$E_\alpha f(x^{x^o})=\int_{\partial \R^n_+}\frac{f(\zeta^{x^o})}{|x^{x^o}-\zeta^{x^o}|^{n-\alpha}}\big(\frac{2}{|\zeta-x^o|}\big)^{2(n-1)}d\zeta,$$
thus
$$
f^{p-1}(y^{x^o})=\int_{\R^n_+}\frac{\big(\int_{\partial \R^n_+}\frac{f(\zeta^{x^o})}{|x-\zeta|^{n-\alpha}}\big(\frac{2}{|\zeta-x^o|}\big)^{n+\alpha-2}d\zeta\big)^{t^\prime-1}}{|x-y|^{n-\alpha}}
\big(\frac{|y-x^o|}{2}\big)^{n-\alpha}\big(\frac{2}{|x-x^o|}\big)^{(\alpha-n)t^\prime+2n}dx.
$$
That is
\begin{eqnarray}\label{equ-UHS-sub}\nonumber
&&\big((\frac{2}{|y-x^o|})^{\frac{n-\alpha}{p-1}}f(y^{x^o})\big)^{p-1}\\
&=&\int_{\R^n_+}\frac{\big(\int_{\partial \R^n_+}\frac{f(\zeta^{x^o})}{|x-\zeta|^{n-\alpha}}\big(\frac{2}{|\zeta-x^o|}\big)^{n+\alpha-2}d\zeta\big)^{t^\prime-1}}{|x-y|^{n-\alpha}}
\big(\frac{2}{|x-x^o|}\big)^{(\alpha-n)t^\prime+2n}dx.
\end{eqnarray}
Denote $$F(y):=\big(\frac{2}{|y-x^o|}\big)^{\frac{2(n-1)}{p}}f(y^{x^o}),~\overline{u}(y):=F^{p-1}(y)$$
and
$$\overline{v}(x)=\int_{\partial \R^n_+}\frac{F(\zeta)}{|x-\zeta|^{n-\alpha}}\big(\frac{2}{|\zeta-x^o|}\big)^{n+\alpha-2-\frac{2(n-1)}{p}}\big(\frac{2}{|x-x^o|}\big)^{n+\alpha-\frac{2n}{t}}d\zeta.$$ For simplicity, we also denote $\kappa:=t^\prime-1<0,~ \theta:=\frac{1}{p-1}<0,~ s_1^\prime:=n+\alpha-2-\frac{2(n-1)}{p}<0,~  s_2^\prime:=n+\alpha-\frac{2n}{t}<0$ since $0<p<\frac{2(n-1)}{n+\alpha-2}, 0<t<\frac{2n}{n+\alpha}$. Thus the single equation \eqref{equ-UHS-sub} can be rewritten as
\begin{equation}\label{equ1-UHS-sub}
\begin{cases}
\overline{u}(y)=\int_{\R^n_+}\frac{\overline{v}^\kappa(x)}{|x-y|^{n-\alpha}}
\big(\frac{2}{|x-x^o|}\big)^{s^\prime_2}\big(\frac{2}{|y-x^o|}\big)^{s^\prime_1}dx, ~y\in\partial \R^n_+,\\
\overline{v}(x)=\int_{\partial \R^n_+}\frac{\overline{u}^\theta(y)}{|x-y|^{n-\alpha}}\big(\frac{2}{|x-x^o|}\big)^{s^\prime_2}\big(\frac{2}{|y-x^o|}\big)^{s^\prime_1}dy,~x\in \R^n_+,
\end{cases}
\end{equation}
where $\overline{u}\in L^{\theta+1}(\partial \R^n_+), \ \overline{v}\in L^{t^\prime}(\R^n_+).$

%
For simplicity, let $u(y)=\overline{u}(y)|y-x^o|^{s^\prime_1}, $ and $ v(x)=\overline{v}(x)|x-x^o|^{s^\prime_2}.$ Thus, up to a constant, we shall  just consider
\begin{equation}\label{equ2-UHS-sub}
\begin{cases}
u(y)=\int_{\R^n_+}\frac{v^\kappa(x)}{|x-y|^{n-\alpha}}\frac{1}{|x-x^o|^{s_2}}
dx, ~y\in\partial \R^n_+,\\
v(x)=\int_{\partial \R^n_+}\frac{u^\theta(y)}{|x-y|^{n-\alpha}}\frac{1}{|y-x^o|^{s_1}}dy,~x\in \R^n_+,
\end{cases}
\end{equation}
where $s_1:=n+\alpha-2-\frac{n-\alpha}{p-1}>0,~  s_2:=(\alpha-n)t^\prime+2n>0.$


In the rest of this section,  for any positive Lebesgue measurable  solutions $(u,v)$ to system \eqref{equ2-UHS-sub}, we always assume that $u(y)=\overline{u}(y)|y-x^o|^{s^\prime_1},v(x)=\overline{v}(x)|x-x^o|^{s^\prime_2}$ with  $\overline{u}(y)\in L^{\theta+1}(\partial \R^n_+),\overline{v}(x)\in L^{t^\prime}(\R^n_+)$.

To classify all positive solutions to  system \eqref{equ2-UHS-sub}, we need some lemmas.
\begin{lm}\label{lm-UHS-mmp-es}
Assume that $\alpha>n$ and $\theta, \kappa<0$.  Let $(u,v)$ be a pair of positive Lebesgue measurable  solutions to system \eqref{equ2-UHS-sub}. Then

(\Rmnum{1}) $\int_{\partial \R^n_+} (1+|y|^{\alpha-n})\frac{u^\theta(y)}{|y-x^o|^{s_1}}dy<\infty,~\int_{\R^n_+} (1+|x|^{\alpha-n})\frac{v^\kappa(x)}{|x-x^o|^{s_2}}dx<\infty;$

(\Rmnum{2}) it holds
$$0<a:=\lim\limits_{|y|\to\infty}|y|^{n-\alpha}u(y)=\int_{\R^n_+} \frac{v^\kappa(x)}{|x-x^o|^{s_2}}dx<\infty,$$
and
$$0<b:=\lim\limits_{|x|\to\infty}|x|^{n-\alpha}v(x)=\int_{\partial \R^n_+} \frac{u^\theta(y)}{|y-x^o|^{s_1}}dy<\infty;$$

(\Rmnum{3}) there exist some constants $C_1, C_2$ such that

$$\frac{1+|y|^{\alpha-n}}{C_1}\le u(y)\le C_1(1+|y|^{\alpha-n}), ~\forall ~y\in\partial \R^n_+,$$
and
$$\frac{1+|x|^{\alpha-n}}{C_2}\le v(x)\le C_2(1+|x|^{\alpha-n}), ~\forall ~x\in \R^n_+.$$
\end{lm}
\begin{proof}
We use the similar idea from the proof of Lemma 5.1 in \cite{Li}.

{First of all,  it is easy to know that $u,v\not\equiv\infty$ from the integrability assumption}, thus $$meas\{y\in \partial \R^n_+: u(y)<\infty\}>0, \ \ meas\{x\in \R^n_+: v(x)<\infty\}>0.$$

Take $y^1, y^2\in \partial\R^n_+, |y^1-y^2|=\delta$ such that
$$u(y^1)<\infty, \ u(y^2)<\infty.$$
It is easy to know that there exists $C>0$ such that $$|x-y^1|^{\alpha-n}\ge C(1+|x|^{\alpha-n}),~\forall~x\in \R^n_+\setminus B(y^1,\frac{\delta}{2}),$$ and
$$|x-y^2|^{\alpha-n}\ge C(1+|x|^{\alpha-n}),~\forall~x\in \R^n_+\setminus B(y^2,\frac{\delta}{2}).$$
Thus
\begin{eqnarray*}
\infty &>& u(y^1)+u(y^2)\\
&\ge& C\int_{\R^n_+ \setminus B(y^1,\frac{\delta}{2})}\frac{v^\kappa(x)}{|x-y^1|^{n-\alpha}}
\frac{1}{|x-x^o|^{s_2}}dx+C\int_{\R^n_+ \setminus B(y^2,\frac{\delta}{2})}\frac{v^\kappa(x)}{|x-y^2|^{n-\alpha}}
\frac{1}{|x-x^o|^{s_2}}dx\\
&\ge & C \int_{\R^n_+}(1+|x|^{\alpha-n})\frac{v^\kappa(x)}{|x-x^o|^{s_2}}dx.
\end{eqnarray*}
So $$\int_{\R^n_+}(1+|x|^{\alpha-n})\frac{v^\kappa(x)}{|x-x^o|^{s_2}}dx<\infty.$$ Similarly, we have
$$\int_{\partial \R^n_+} (1+|y|^{\alpha-n})\frac{u^\theta(y)}{|y-x^o|^{s_1}}dy<\infty.$$ Thus (\Rmnum 1) holds.

For $|y|\ge 1$, it holds
$$\frac{|x-y|^{\alpha-n}}{|y|^{\alpha-n}}\le 2^{\alpha-n}(1+|x|^{\alpha-n}).$$  By the  Lebesgue Dominated Convergence Theorem and (\Rmnum 1), we have
$$
a=\lim\limits_{|y|\to\infty}|y|^{n-\alpha}u(y)=
\lim\limits_{|y|\to\infty}\int_{\R^n_+}\frac{|x-y|^{\alpha-n}}{|y|^{\alpha-n}}
\frac{v^\kappa(x)}{|x-x^o|^{s_2}}dx
=\int_{\R^n_+} \frac{v^\kappa(x)}{|x-x^o|^{s_2}}dx<\infty.$$
Similarly,
$$b=\lim\limits_{|x|\to\infty}|x|^{n-\alpha}v(x)=\int_{\partial \R^n_+} \frac{u^\theta(y)}{|y-x^o|^{s_1}}dy<\infty.$$
We thus obtain  (\Rmnum 2).

Take $R>1$ and some measurable set $E$ such that $$E\subset \{x\in \R^n_+,v(x)<R\}\cap B(x^o, R)$$ with $|E|>\frac{1}{R}$.  Then for any $y\in \partial\R^n_+$,
\begin{eqnarray*}
u(y)&=&\int_{\R^n_+}\frac{v^\kappa(x)}{|x-y|^{n-\alpha}}
\frac{1}{|x-x^o|^{s_2}}dx\\
&\ge&\int_E \frac{v^\kappa(x)}{|x-y|^{n-\alpha}}
\frac{1}{|x-x^o|^{s_2}}dx\\
&\ge& R^{\kappa-s_2}\int_E |x-y|^{\alpha-n}dx.
\end{eqnarray*}
So
\begin{eqnarray*}
\lim_{|y|\to\infty}\frac{u(y)}{(1+|y|^{\alpha-n})}&\ge& C\lim_{|y|\to\infty}\frac{R^{\kappa-s_2}}{(1+|y|^{\alpha-n})}\int_E |x-y|^{\alpha-n}dx\\
&\ge& C R^{\kappa-s_2}.
\end{eqnarray*}
That is, there exists $M>0$ large enough such that, for any $|y|\ge M$, it holds
\begin{equation}\label{est1-uv-UHS}
u(y)\ge C(1+|y|^{\alpha-n}).
\end{equation}
On the other hand, we also have, via equation \eqref{equ2-UHS-sub}, that there exists $C_0>0$ such that
\begin{equation}\label{est2-uv-UHS}
\inf_{B(0,M)\cap \partial \R^n_+}u(y)\ge C_0.
\end{equation}
%
Thus the lower bound inequalities in  (\Rmnum 3) hold.  By using (\Rmnum 2), we can easily obtain the upper bound inequalities in  (\Rmnum 3).
\end{proof}

\begin{lm}\label{lm-UHS-regu}
Assume that $\alpha>n$ and $\theta, \kappa<0$.  Let $(u,v)$ be a pair of positive Lebesgue measurable  solutions to system \eqref{equ2-UHS-sub}. Then  $u(y), y\in \partial \R^n_+$ is differentiable in $y_i, i=1,2,\cdots,n-1$, and $v(x), x\in \overline{\R^n_+}$ is differentiable in $x_i, i=1,2,\cdots,n$.
\end{lm}
\begin{proof}
We first show that  $u(y), y\in \partial \R^n_+$ and $v(x), x\in \overline{\R^n_+}$ are continuous.

It is easy to see that for any $y^0\in \partial \R^n_+, y\in B(y^0, \delta)\cap \partial \R^n_+$ with $\delta>0$ small enough,
$$u(y)\le C(y^0,\delta)\int_{\R^n_+} (1+|x|^{\alpha-n})\frac{v^\kappa(x)}{|x-x^o|^{s_2}}dx.$$
By using (\Rmnum 1) in Lemma \ref{lm-UHS-mmp-es} and Lebesgue Dominated Convergence Theorem, $u(y)$ is continuous in $y^0$.
Similarly, we can show that  $v(x), x\in \overline{\R^n_+}$ is continuous.

\smallskip

Next,  we will show that $u(y)$ is differentiable in $y_i, i=1,2,\cdots, n-1$, $v(x)$ is differentiable in $x_i, i=1,2,\cdots,n$.

Take $y\in \partial \R^n_+$ and $R>0$ large enough such that $y\in B(0, R)\cap \partial \R^n_+$ and write
\begin{eqnarray*}u(y)
&=&\int_{\R^n_+}\frac{v^\kappa(x)}{|x-y|^{n-\alpha}}
\frac{1}{|x-x^o|^{s_2}}dx\\
&=&\int_{B(0,2R)\cap \R^n_+}\frac{v^\kappa(x)}{|x-y|^{n-\alpha}}
\frac{1}{|x-x^o|^{s_2}}dx+\int_{\R^n_+\setminus B(0,2R)}\frac{v^\kappa(x)}{|x-y|^{n-\alpha}}
\frac{1}{|x-x^o|^{s_2}}dx\\
&:=&I_1(y)+I_2(y).
\end{eqnarray*}
It is easy to see that $I_2(y)$ is differentiable in $y_i, i=1,2,\cdots,n-1$. Now we prove $I_1(y)$ is differentiable in $y_1.$

If $\alpha-n\ge 1$, for fixed $y^0=(y_1^0,y_2^0,\cdots,y_{n-1}^0,0)\in B(0, R)\cap \partial \R^n_+$ and $y=(y_1,y_2^0,\cdots,y_{n-1}^0,0)\in B(0, R)\cap \partial \R^n_+$ with $y_1\neq y_1^0$, we have
$$\bigg|\frac{|x-y|^{\alpha-n}-|x-y^0|^{\alpha-n}}{y_1-y_1^0}\frac{v^\kappa(x)}{|x-x^o|^{s_2}}\bigg|\le C\frac{v^\kappa(x)}{|x-x^o|^{s_2}},~\forall~ x\in B(0, 2R)\cap \R^n_+,$$
where $C$ is independent of $x,y$.
Thus Lebesgue Dominated Convergence Theorem implies that $\frac{\partial I_1(y)}{\partial y_1}\mid _{y=y^0}$ exists.

If $0<\alpha-n<1$, we need the following simple inequality:
$$|a_1^\beta-a_2^\beta|\le a_2^{\beta-1}|a_1-a_2|$$
for any $a_1,a_2>0, 0<\beta<1$. Thus for fixed $y^0=(y_1^0,y_2^0,\cdots,y_{n-1}^0,0)\in B(0, R)\cap \partial \R^n_+$ and $y=(y_1,y_2^0,\cdots,y_{n-1}^0,0)\in B(0, R)\cap \partial \R^n_+$ with $y_1\neq y_1^0$, we have
$$\bigg|\frac{|x-y|^{\alpha-n}-|x-y^0|^{\alpha-n}}{y_1-y_1^0}\frac{v^\kappa(x)}{|x-x^o|^{s_2}}\bigg|\le C|x-y^0|^{\alpha-n-1}\frac{v^\kappa(x)}{|x-x^o|^{s_2}},~\forall~ x\in B(0, 2R)\cap \R^n_+,$$
where $C$ is independent of $x,y$. It is easy to see that
$$\int_{B(0,2R)\cap \R^n_+} |x-y^0|^{\alpha-n-1}\frac{v^\kappa(x)}{|x-x^o|^{s_2}}dx<\infty.$$
Thus Lebesgue Dominated Convergence Theorem implies that $\frac{\partial I_1(y)}{\partial y_1}\mid _{y=y^0}$ exists.

Similar arguments yield that $\frac{\partial I_1(y)}{\partial y_i}\mid _{y=y^0}$ exists for $i=2, \cdots, n-1$, as well as
 $v(x)$ is differentiable in $x_i, i=1,2,\cdots,n$.
\end{proof}

\begin{lm}\label{estD-UHS-uv}
Assume that $\alpha>n$ and $\theta, \kappa<0$.  Let $(u,v)$ be a pair of positive Lebesgue measurable  solutions to system \eqref{equ2-UHS-sub}. Then 
we have
\begin{eqnarray}\label{estD1-UHS-uv}
0<a(\alpha-n)&=&\lim\limits_{|y|\to +\infty} |y|^{n-\alpha}\big(\nabla u(y)\cdot y+\frac{n-\alpha}{2}u(y)\big),\\\label{estD2-UHS-uv}
0<b(\alpha-n)&=&\lim\limits_{|x|\to +\infty} |x|^{n-\alpha}\big(\nabla v(x)\cdot x+\frac{n-\alpha}{2}v(x)\big),\end{eqnarray}
where $a,b$ are the same numbers in  Lemma \ref{lm-UHS-mmp-es}.
\end{lm}
\begin{proof}
We only prove \eqref{estD1-UHS-uv}. By using Lemma \ref{lm-UHS-regu},   letting $|y|\ge 1$
\begin{eqnarray*}
&&\bigg||y|^{n-\alpha}\big(\nabla u(y)\cdot y+\frac{n-\alpha}{2}u(y)\big)\bigg|\\
&=&(\alpha-n)\bigg||y|^{n-\alpha}\int_{\R^n_+} |x-y|^{\alpha-n-2}\big(\sum_{i=1}^n(y_i-x_i)y_i-\frac{|x-y|^2}{2}\big)\frac{v^\kappa(x)}{|x-x^o|^{s_2}}dx\bigg|\\
&=&(\alpha-n)|y|^{n-\alpha}\int_{\R^n_+} |x-y|^{\alpha-n-2}||y|^2-|x|^2|\frac{v^\kappa(x)}{|x-x^o|^{s_2}}dx\\
&\le&C(\alpha-n)\int_{\R^n_+} (1+|x|^{\alpha-n})\frac{v^\kappa(x)}{|x-x^o|^{s_2}}dx.
\end{eqnarray*}
Thus Lebesgue Dominated Convergence Theorem implies that
$$\lim\limits_{|y|\to +\infty} |y|^{n-\alpha}(\nabla u(y)\cdot y-\frac{u(y)}{2})=(\alpha-n)\int_{\R^n_+} \frac{v^\kappa(x)}{|x-x^o|^{s_2}}dx=a(\alpha-n).$$
\end{proof}

Next, we shall establish the symmetric properties about the  positive, measurable solutions to system \eqref{equ2-UHS-sub}.
\begin{prop}\label{thm-sym-UHS-sub}
Let $(u,v)$ be a pair of positive Lebesgue measurable solutions to system \eqref{equ2-UHS-sub}.
For $\alpha>n$, if $0<p<\frac{2(n-1)}{n+\alpha-2}, 0<t<\frac{2n}{n+\alpha}$, then $u(y), y\in \partial \R^n_+$ and $v(x), x\in \R^n_+$ must be symmetric with respect to $x_n-$axis.
\end{prop}
We shall relegate the proof late.  But first, from this, we shall obtain the extremal functions and the best constant for the reversed HLS in a ball with subcrtical exponents.

\begin{prop}\label{thm-sol-UHS-sub}
For $\alpha>n$, if $0<p<\frac{2(n-1)}{n+\alpha-2}, 0<t<\frac{2n}{n+\alpha}$, then the positive solution to \eqref{equ-sub}  must be constant function. Thus $\xi_\alpha$ is given by \eqref{9-19-1}.

\end{prop}
\begin{proof}
From Proposition \ref{thm-sym-UHS-sub} we know that $f(y^1)=f(y^2)$ if $|y^1-x^o|=|y^2-x^o|$, $y^1, y^2\in \partial B^n,$ where we recall $x^o=(0,-2)$. But actually we can prove $f(y^1)=f(y^2)$ for any $y^1, y^2\in \partial B^n.$

In fact, for any two different points $y^1, y^2\in \partial B^n$, we can choose $x^1\in \partial B^n$ such that $|y^1-x^1|=|y^2-x^1|.$ There exists a rotation transformation $\cal{T}$, such that $x^o=\cal{T}x^1$.
We use the transformation (from $z\in B^n$ to $y\in \R^n_+$)
$$y=\widehat{z^{x^1}}:=\frac{2^2(z^\prime-x^o)}{|z^\prime-x^o|^2}+x^o,$$
where $z\in B^n, z^\prime=\cal{T} z$.
Under this transformation, equation \eqref{equ-sub} will be changed to the same system \eqref{equ2-UHS-sub}. Thus we conclude $f(y^1)=f(y^2)$ since $|y^1-x^1|=|y^2-x^1|$.

The infimum can be computed easily.
\end{proof}




\subsection{Symmetry via the method of moving planes}
Now we prove Proposition \ref{thm-sym-UHS-sub} via the method of moving planes.

Denote $x=\{x_1,x_2,\cdots,x_n\}\in\R^n$. Let $$T_\lambda=\{x\in\R^n|x_1=\lambda\},x_\lambda=\{2\lambda-x_1,x_2,\cdots,x_n\},$$
$$\Sigma_{\lambda,n}=\{x\in\R^n_+|x_1\le\lambda\}, \Sigma_{\lambda,n-1}=\{x\in\partial \R^n_+|x_1\le\lambda\},$$
and denote
$$u_\lambda (y):=u(y_\lambda),~\mbox{for}~y\in \partial\R^n_+,$$
$$v_\lambda (x):=v(x_\lambda),~\mbox{for}~x\in \R^n_+.$$


\begin{lm}\label{lm-UHS-mmp-UV}
Let $(u,v)$ be a pair of positive solutions to system \eqref{equ2-UHS-sub}. Then we have
\begin{equation}\label{UHS-mmp-U}
u(y)-u_\lambda (y)=\int_{\Sigma_{\lambda,n}}(|x_\lambda-y|^{\alpha-n}-|x-y|^{\alpha-n})(\frac{v_\lambda^\kappa(x)}{|x_\lambda-x^o|^{s_2}}-\frac{v^\kappa(x)}{|x-x^o|^{s_2}})dx, ~y\in\partial \R^n_+,
 \end{equation}
and
\begin{equation}\label{UHS-mmp-V}
v(x)-v_\lambda (x)=\int_{\Sigma_{\lambda,n-1}}(|y_\lambda-x|^{\alpha-n}-|y-x|^{\alpha-n})(\frac{u_\lambda^\theta(y)}{|y_\lambda-x^o|^{s_1}}-\frac{u^\theta(y)}{|y-x^o|^{s_1}})dy,~x\in \R^n_+.
\end{equation}
\end{lm}
\begin{proof}
We only prove the first one. Direct computation yields
\begin{eqnarray*}
u(y)
&=&\int_{\Sigma_{\lambda,n}}\frac{v^\kappa(x)}{|x-y|^{n-\alpha}}
\frac{1}{|x-x^o|^{s_2}}dx+\int_{\R^n_+\backslash\Sigma_{\lambda,n}}\frac{v^\kappa(x)}{|x-y|^{n-\alpha}}
\frac{1}{|x-x^o|^{s_2}}dx\\
&=&\int_{\Sigma_{\lambda,n}}\frac{v^\kappa(x)}{|x-y|^{n-\alpha}}
\frac{1}{|x-x^o|^{s_2}}dx+\int_{\Sigma_{\lambda,n}}\frac{v_\lambda^\kappa(\xi)}{|\xi_\lambda-y|^{n-\alpha}}
\frac{1}{|\xi_\lambda-x^o|^{s_2}}d\xi,
\end{eqnarray*}
and
\begin{eqnarray*}
u_\lambda(y)
&=&\int_{\Sigma_{\lambda,n}}\frac{v^\kappa(x)}{|x-y_\lambda|^{n-\alpha}}
\frac{1}{|x-x^o|^{s_2}}dx+\int_{\R^n_+\backslash\Sigma_{\lambda,n}}\frac{v^\kappa(x)}{|x-y_\lambda|^{n-\alpha}}
\frac{1}{|x-x^o|^{s_2}}dx\\
&=&\int_{\Sigma_{\lambda,n}}\frac{v^\kappa(x)}{|x-y_\lambda|^{n-\alpha}}
\frac{1}{|x-x^o|^{s_2}}dx+\int_{\Sigma_{\lambda,n}}\frac{v_\lambda^\kappa(\xi)}{|\xi_\lambda-y_\lambda|^{n-\alpha}}
\frac{1}{|\xi_\lambda-x^o|^{s_2}}d\xi.
\end{eqnarray*}
Notice that $|x_\lambda-y|=|x-y_\lambda|$ and $|x_\lambda-y_\lambda|=|x-y|$,  \eqref{UHS-mmp-U} follows from  above equalities.
\end{proof}

\begin{lm}\label{lm-UHS-mmp-start}
Let $(u,v)$ be a pair of positive Lebesgue measurable solutions to system \eqref{equ2-UHS-sub}. Assume that $\alpha>n$, $0<p<\frac{2(n-1)}{n+\alpha-2}, 0<t<\frac{2n}{n+\alpha}$.  Then, for sufficiently negative $\lambda$, we have
\begin{eqnarray*}
u(y)&\ge & u_\lambda(y),~\forall ~y\in\Sigma_{\lambda,n-1},\\
v(x)&\ge & v_\lambda (x),~\forall ~x\in\Sigma_{\lambda,n}.
\end{eqnarray*}
\end{lm}

\begin{proof} We shall follow the proof of Lemma 5.4 in \cite{Li}.

By  Lemma \ref{estD-UHS-uv}, we know that  there exists $\lambda_0<0$ sufficiently negative such that, for  $y_1<\lambda_0$,
\begin{eqnarray*}
\nabla\big(|y|^{\frac{n-\alpha}{2}}u(y)\big)\cdot y=|y|^{\frac{n-\alpha}{2}}\big(\nabla u(y)\cdot y+\frac{n-\alpha}{2}u(y)\big)>0.
\end{eqnarray*}

For $y\in\Sigma_{\lambda,n-1}$, if $|y_1|\ge-\lambda_0$ and $|2\lambda-y_1|\ge-\lambda_0$, we then have
$$|y|^{\frac{n-\alpha}{2}}u(y)>|y_\lambda|^{\frac{n-\alpha}{2}}u_\lambda(y),$$
which implies
$$u(y)>u_\lambda(y).$$

On the other hand,  for $y\in\Sigma_{\lambda,n-1}$, if $|2\lambda-y_1|\le -\lambda_0$,  by Lemma \ref{lm-UHS-mmp-es} (\Rmnum{3}),  we can take $|\lambda|$ large enough such that
$$u(y)\ge\frac{1+|y_1|^{\alpha-n}}{C_1}\ge\frac{1+|\lambda|^{\alpha-n}}{C_1}\ge C_1(1+|\lambda_0|^{\alpha-n})\ge C_1(1+|2\lambda-y_1|^{\alpha-n})\ge u_\lambda(y).$$

Similarly, one can get the inequality  for $v(x)$.
\end{proof}
\smallskip

Let
$$\overline{\lambda}=\sup\{\mu<0 \ : \ u(y)\ge u_\lambda(y)~\text{and}~ v(x)\ge v_\lambda (x),~\forall ~y\in\Sigma_{\lambda,n-1},~\forall ~x\in\Sigma_{\lambda,n}, \forall \lambda\in(-\infty, \mu)\}.$$

\begin{lm}\label{lm-UHS-mmp-infy}
$\overline{\lambda}=0$.
\end{lm}
\begin{proof}
Assume by contradiction that $\overline{\lambda}<0$.

Since $u,\ v$ are continuous, we have
\begin{eqnarray*}
u(y)&\ge & u_{\overline{\lambda}}(y),~\forall ~y\in\Sigma_{\overline{\lambda},n-1},\\
v(x)&\ge & v_{\overline{\lambda}}(x),~\forall ~x\in\Sigma_{\overline{\lambda},n}.
\end{eqnarray*}
Then by Lemma \ref{lm-UHS-mmp-UV},  we have
\begin{eqnarray*}
u(y)&> & u_{\overline{\lambda}}(y),~\forall ~y\in\Sigma_{\overline{\lambda},n-1},\\
v(x)&> & v_{\overline{\lambda}}(x),~\forall ~x\in\Sigma_{\overline{\lambda},n}.
\end{eqnarray*}
Otherwise, if there exists $y^0\in\Sigma_{\overline{\lambda},n-1}$ such that $u(y^0)=u_{\overline{\lambda}}(y^0)$, then
$$v(x)\equiv v_{\overline{\lambda}}(x)\equiv 0,~\forall ~ x\in \R^n_+,$$
since $|x_{\overline{\lambda}}-y|^{\alpha-n}-|x-y|^{\alpha-n}>0, |x_{\overline{\lambda}}-x^o|<|x-x^o|$ (note that $\overline{\lambda}<0$). This is impossible.

Next we prove that, there exists $\epsilon>0$ such that, for any $\lambda\in(\overline{\lambda},\overline{\lambda}+\epsilon)$,
\begin{eqnarray*}
u(y)&\ge & u_\lambda(y),~\forall ~y\in\Sigma_{\lambda,n-1},\\
v(x)&\ge & v_\lambda (x),~\forall ~x\in\Sigma_{\lambda,n}.
\end{eqnarray*}

We divide the proof into two steps.

Step 1. We show that there is some $0<\epsilon_1<1 $ satisfying $\overline{\lambda}+\epsilon_1<0$, such that for any $0<\epsilon<\epsilon_1$, $\overline{\lambda}\le\lambda\le \overline{\lambda}+\epsilon$,  if $y_1\le\overline{\lambda}-1,  \ x_1\le\overline{\lambda}-1$, then
\begin{equation*}
u(y)-u_\lambda(y)\ge\frac{\epsilon_1}{2}|y|^{\alpha-n-1},~v(x)-v_\lambda (x)\ge\frac{\epsilon_1}{2}|x|^{\alpha-n-1}.
\end{equation*}

 By \eqref{UHS-mmp-U} and Fatou's Lemma, if $y_1\le\overline{\lambda}-1$, for any fixed $R>0$ large enough
\begin{eqnarray*}
&&\liminf\limits_{|y|\to\infty} |y|^{n-\alpha+1} (u(y)-u_{\overline{\lambda}}(y))\\
&\ge &\int_{\Sigma_{{\overline{\lambda}},n}}\liminf\limits_{|y|\to\infty} |y|^{n-\alpha+1}\big(|\xi_{\overline{\lambda}}-y|^{\alpha-n}-|\xi-y|^{\alpha-n}\big)
\big(\frac{v_{\overline{\lambda}}^\kappa(\xi)}{|\xi_{\overline{\lambda}}-x^o|^{s_2}}-\frac{v^\kappa(\xi)}{|\xi-x^o|^{s_2}}\big)d\xi\\
&\ge &\int_{\Sigma_{{\overline{\lambda}},n}\cap B(0,R)}\liminf\limits_{|y|\to\infty} |y|^{n-\alpha+1}\big(|\xi_{\overline{\lambda}}-y|^{\alpha-n}-|\xi-y|^{\alpha-n}\big)
\big(\frac{v_{\overline{\lambda}}^\kappa(\xi)}{|\xi_{\overline{\lambda}}-x^o|^{s_2}}-\frac{v^\kappa(\xi)}{|\xi-x^o|^{s_2}}\big)d\xi\\
&\ge & c_0>0,
\end{eqnarray*}
since $v(x)>v_{\overline{\lambda}}(x),~\forall ~x\in\Sigma_{\overline{\lambda},n}$.
 Hence there exists $0<\epsilon_2<1$ such that
$$u(y)-u_{\overline{\lambda}}(y)\ge \epsilon_2 |y|^{\alpha-n-1},$$
if $y_1\le\overline{\lambda}-1$.

On the other hand, since $|\nabla u(y)| \le C |y|^{\alpha-n-1}$ by Lemma 2.8 and Lemma 2.6 (III), we can take $0<\epsilon_3<\epsilon_2$ small enough such that if $y_1\le\overline{\lambda}-1$, $\overline{\lambda}\le \lambda\le \overline{\lambda}+\epsilon_3$, then
$$|u_\lambda(y)-u_{\overline{\lambda}}(y|\le \frac{\epsilon_2}{2}|y|^{\alpha-n-1}.$$

Thus
$$u(y)-u_\lambda(y)\ge \frac{\epsilon_3}{2} |y|^{\alpha-n-1},$$
if $y_1\le\overline{\lambda}-1, \overline{\lambda}\le \lambda\le \overline{\lambda}+\epsilon_3.$

Similar argument leads to the estimate for $v(x)$.


Step 2. We show that there exists $0<\epsilon_*<\epsilon_1$, such that, for any $0<\epsilon<\epsilon_*$, $\overline{\lambda}\le \lambda\le \overline{\lambda}+\epsilon$, if $\overline{\lambda}-1\le y_1\le \lambda, \overline{\lambda}-1\le x_1\le \lambda$, then
\begin{eqnarray*}
u(y)&\ge & u_\lambda(y),\\
v(x)&\ge & v_\lambda (x).
\end{eqnarray*}

It follows from Lemma \ref{estD-UHS-uv} that there exists $R_1\ge 4(|\overline{\lambda}|+1)$ large enough such that for any $y\in (\Sigma_{\lambda,n-1}\backslash \Sigma_{\overline{\lambda}-1,n-1})\backslash B(0, \frac{R_1}{2}),$
\begin{eqnarray*}
\nabla\big(|y|^{\frac{n-\alpha}{2}}u(y)\big)\cdot y=|y|^{\frac{n-\alpha}{2}}\big(\nabla u(y)\cdot y+\frac{n-\alpha}{2}u(y)\big)>0.
\end{eqnarray*}
Based on this,  we know that if $|y|\ge R_1$, thus $|y_\lambda|\ge\frac{R_1}{2}$, then, for $\overline{\lambda}\le\lambda\le \overline{\lambda}+\epsilon_1<0$,
$$|y|^{\frac{n-\alpha}{2}}u(y)>|y_\lambda|^{\frac{n-\alpha}{2}}u_\lambda(y).$$
Therefore,
\begin{equation}\label{equ1-UHS-mmp-infy}
u(y)>u_\lambda(y),~\mbox{if}~|y|\ge R_1.
\end{equation}

Next, we consider $y\in (\Sigma_{\lambda,n-1}\backslash \Sigma_{\overline{\lambda}-1,n-1})\cap B(0, R_1)$.  For convenience, in the following we denote $K(\lambda, \xi, y):=|\xi_\lambda-y|^{\alpha-n}-|\xi-y|^{\alpha-n}$.
First, for $0<\overline{\epsilon}<\epsilon_1$, take $R_0\ge 2R_1$ large enough such that
$$\frac{1}{{R_0}^{n+1}}\le \overline{\epsilon}.$$
Then for any $y\in (\Sigma_{\lambda,n-1}\backslash \Sigma_{\overline{\lambda}-1,n-1})\cap B(0, R_1),\xi\in (\Sigma_{\lambda,n}\backslash \Sigma_{\overline{\lambda}-1,n})\backslash B(0, R_0),$
\begin{eqnarray*}
K(\lambda, \xi, y)&=&|\xi-y_\lambda|^{\alpha-n}-|\xi-y|^{\alpha-n}\\
&\le& C\max\{|\xi_\lambda-y|^{\alpha-n-1}, |\xi-y|^{\alpha-n-1}\}||\xi-y_\lambda|-|\xi-y||\\
&\le& C|\xi|^{\alpha-n-1}|y_\lambda-y|\\
&\le& C|\xi|^{\alpha-n-1}(|y_1|-|\lambda|).
\end{eqnarray*}
Thus by using Lemma \ref{lm-UHS-mmp-es} (III),
\begin{eqnarray*}
&&|\int_{(\Sigma_{\lambda,n}\backslash \Sigma_{\overline{\lambda}-1,n})\backslash B(0, R_0)}K(\lambda, \xi, y)\big(\frac{v_\lambda^\kappa(\xi)}{|\xi_\lambda-x^o|^{s_2}}-\frac{v^\kappa(\xi)}{|\xi-x^o|^{s_2}}\big)d\xi|\\
&\le & C(|y_1|-|\lambda|)|\int_{(\Sigma_{\lambda,n}\backslash \Sigma_{\overline{\lambda}-1,n})\backslash B(0, R_0)}|\xi|^{\alpha-n-1}\cdot \frac{|\xi|^{\kappa(\alpha-n)}}{|\xi|^{s_2}}d\xi|\\
&\le &C(|y_1|-|\lambda|) |\int_{\R^n\backslash B(0, R_0)}|\xi|^{\kappa(\alpha-n)+\alpha-n-1-s_2}d\xi|\\
&=& C(|y_1|-|\lambda|)\int_{R_0}^\infty r^{\kappa(\alpha-n)+\alpha-n-1-s_2}\cdot r^{n-1}dr\\
&=& C(|y_1|-|\lambda|)\frac{1}{{R_0}^{n+1}}\le C \overline{\epsilon} (|y_1|-|\lambda|)
\end{eqnarray*}
 for some positive constant $C$.  Hence
\begin{eqnarray*}
&&u(y)-u_\lambda(y)\\
&\ge&\int_{(\Sigma_{\lambda,n}\backslash \Sigma_{\overline{\lambda}-1,n})}K(\lambda, \xi, y)\big(\frac{v_\lambda^\kappa(\xi)}{|\xi_\lambda-x^o|^{s_2}}-\frac{v^\kappa(\xi)}{|\xi-x^o|^{s_2}}\big)d\xi\\
&&+\int_{(\Sigma_{\overline{\lambda}-2,n}\backslash \Sigma_{\overline{\lambda}-3,n})}K(\lambda, \xi, y)\big(\frac{v_\lambda^\kappa(\xi)}{|\xi_\lambda-x^o|^{s_2}}-\frac{v^\kappa(\xi)}{|\xi-x^o|^{s_2}}\big)d\xi\\
&\ge&-C\overline{\epsilon}(|y_1|-|\lambda|)+\int_{(\Sigma_{\lambda,n}\backslash \Sigma_{\overline{\lambda}-1,n})\cap B(0, R_0)}K(\lambda, \xi, y)\big(\frac{v_\lambda^\kappa(\xi)}{|\xi_\lambda-x^o|^{s_2}}-\frac{v^\kappa(\xi)}{|\xi-x^o|^{s_2}}\big)d\xi\\
&&+\int_{(\Sigma_{\overline{\lambda}-2,n}\backslash \Sigma_{\overline{\lambda}-3,n})\cap B(0, R_0)}K(\lambda, \xi, y)\big(\frac{v_\lambda^\kappa(\xi)}{|\xi_\lambda-x^o|^{s_2}}-\frac{v^\kappa(\xi)}{|\xi-x^o|^{s_2}}\big)d\xi\\
&:=&-C\overline{\epsilon} (|y_1|-|\lambda|)+A+B.
\end{eqnarray*}
We rewrite $A$ as
\begin{eqnarray*}
A&=&\int_{(\Sigma_{\lambda,n}\backslash \Sigma_{\overline{\lambda},n})\cap B(0, R_0)}K(\lambda, \xi, y)\big(\frac{v_\lambda^\kappa(\xi)}{|\xi_\lambda-x^o|^{s_2}}-\frac{v^\kappa(\xi)}{|\xi-x^o|^{s_2}}\big)d\xi\\
&&+\int_{(\Sigma_{\overline{\lambda},n}\backslash \Sigma_{\overline{\lambda}-1,n})\cap B(0, R_0)}K(\lambda, \xi, y)\big(\frac{v_\lambda^\kappa(\xi)}{|\xi_\lambda-x^o|^{s_2}}-\frac{v^\kappa(\xi)}{|\xi-x^o|^{s_2}}\big)d\xi\\
&:=&A_1+A_2.
\end{eqnarray*}
Since $y\in (\Sigma_{\lambda,n-1}\backslash \Sigma_{\overline{\lambda}-1,n-1})\cap B(0, R_1)$, it is easy to see that
\begin{equation*}
|A_1|\le C(R^o) \int_{(\Sigma_{\lambda,n}\backslash \Sigma_{\overline{\lambda},n})\cap B(0, R_0)}K(\lambda, \xi, y)d\xi\le C(R^o,R_1)|y_\lambda-y| \int_{(\Sigma_{\lambda,n}\backslash \Sigma_{\overline{\lambda},n})\cap B(0, R_0)}d\xi.
\end{equation*}
So we can take $0<\epsilon_*<\overline{\epsilon}$ small enough such that for any $\overline{\lambda}\le \lambda\le \overline{\lambda}+\epsilon_*$ ( note the volume of $(\Sigma_{\lambda,n}\backslash \Sigma_{\overline{\lambda},n})\cap B(0, R_0)$ is small ),
\begin{equation*}
|A_1|\le \frac {C}4\overline{\epsilon}(|y_1|-|\lambda|),
\end{equation*}
and
\begin{eqnarray*}
|A_2|&\le & \int_{(\Sigma_{\overline{\lambda},n}\backslash \Sigma_{\overline{\lambda}-1,n})\cap B(0, R_0)}K(\lambda, \xi, y)\bigg|\frac{v_\lambda^\kappa(\xi)}{|\xi_\lambda-x^o|^{s_2}}-\frac{v_{\overline{\lambda}}^\kappa(\xi)}{|\xi_{\overline{\lambda}}-x^o|^{s_2}}\bigg|d\xi\\
&\le& C(R^o,R_1)|y_\lambda-y|\int_{(\Sigma_{\overline{\lambda},n}\backslash \Sigma_{\overline{\lambda}-1,n})\cap B(0, R_0)}\bigg|\frac{v_\lambda^\kappa(\xi)}{|\xi_\lambda-x^o|^{s_2}}-\frac{v_{\overline{\lambda}}^\kappa(\xi)}{|\xi_{\overline{\lambda}}-x^o|^{s_2}}\bigg|d\xi\\
&\le& \frac {C}4\overline{\epsilon}(|y_1|-|\lambda|),
\end{eqnarray*}
since $\bigg|\frac{v_\lambda^\kappa(\xi)}{|\xi_\lambda-x^o|^{s_2}}-\frac{v_{\overline{\lambda}}^\kappa(\xi)}{|\xi_{\overline{\lambda}}-x^o|^{s_2}}\bigg|$ is small uniformly for $\xi\in(\Sigma_{\overline{\lambda},n}\backslash \Sigma_{\overline{\lambda}-1,n})\cap B(0, R_0)$  and   $ \lambda \in [\overline{\lambda},  \overline{\lambda}+\epsilon_*]$.
So we have
$$A\ge -\frac C 2\overline{\epsilon}(|y_1|-|\lambda|).$$
Now we always assume  $\overline{\lambda}\le \lambda\le \overline{\lambda}+\epsilon_*$. Notice that
$$\frac{\partial K(\lambda, \xi, y)}{\partial y_1}\cdot y_1\big\vert_{y_1=\lambda}=(\alpha-n)|\xi-y|^{\alpha-n-2}(2\xi_1-2\lambda)y_1>0$$ if $\xi_1<\lambda$. Since $K(\lambda, \xi, y)\big\vert_{y_1=\lambda}=0$, there exists $\delta>0$ independent of $\overline{\epsilon}, \epsilon_*$, such that
$$K(\lambda, \xi, y)\ge \delta (|y_1|-|\lambda|),
$$ for any $y\in (\Sigma_{\lambda,n-1}\backslash \Sigma_{\overline{\lambda}-1,n-1})\cap B(0, R_1)$ and $\xi$ satisfying $\overline{\lambda}-3\le\xi_1\le\overline{\lambda}-2<\lambda.$   By Step 1, there exists $\delta_1$ independent of $\overline{\epsilon}, \epsilon_*$, such that
$$v_\lambda^\kappa(\xi)-v^\kappa(\xi)\ge\delta_1, ~\forall~\xi\in \Sigma_{\overline{\lambda}-2,n}\backslash \Sigma_{\overline{\lambda}-3,n}.$$ Hence
\begin{equation*}
B\ge\int_{(\Sigma_{\overline{\lambda}-2,n}\backslash \Sigma_{\overline{\lambda}-3,n})\cap B(0, R_0)}K(\lambda, \xi, y)\frac{(v_\lambda^\kappa(\xi)-v^\kappa(\xi))}{|\xi-x^o|^{s_2}}d\xi\ge C\delta_1\delta (|y_1|-|\lambda|)
\end{equation*}
 for some positive constant $C$.

Based on the above, we have for $y\in (\Sigma_{\lambda,n-1}\backslash \Sigma_{\overline{\lambda}-1,n-1})\cap B(0, R_1)$ that
$$u(y)-u_\lambda(y)\ge -C\overline{\epsilon} (|y_1|-|\lambda|)+C\delta_1\delta (|y_1|-|\lambda|)\ge 0,$$which combining with \eqref{equ1-UHS-mmp-infy} ends the proof for $u(y)$ of Step 2, where  $\overline{\lambda}\le \lambda\le \overline{\lambda}+\epsilon_*$.

The proof for $v(x)$ is similar. These contradict the definition of $\overline{\lambda}$, thus implies $\overline{\lambda}=0.$
\end{proof}

{\bf Proof of Proposition \ref{thm-sym-UHS-sub}.}
From Lemma \ref{lm-UHS-mmp-infy}, we know that  both $u(y)$ and $v(x)$ are symmetric with respect to $x_1=0$ since one also can move the plane from positive side of $x_1$  to zero.
Similar argument shows that  $u(y)$ and $v(x)$ are symmetric with respect to $x_2=0, x_3=0, \cdots,x_{n-1}=0$. \hfill$\Box$


\subsection{Proof of Theorem \ref{HLS-ball}}
For any $a>0$ and $b \ge \omega_{n-1}^{1/p-1/2} a$, consider
$$
\inf\{\|E_\alpha f\|_{L^{t^\prime}(B^n)}\ \ : \ \ f \ge 0, \ \ \ \|f\|_{L^p(\partial B^n)}=a, \  \|f\|_{L^2(\partial B^n)}\le b\}.
$$
Letting $v=a^{-1}f$, we know the above infimum is attained by a constant function.
 This indicates that: if $0<p<\frac{2(n-1)}{n+\alpha-2}, 0<t<\frac{2n}{n+\alpha}$, then for any $f\in L^2(\partial B^n)$,
$$
\|E_\alpha f\|_{L^{t^\prime}(B^n)} \ge \xi_\alpha \|f\|_{L^p(\partial B^n)}.$$

For any general nonnegative $f\in L^p(\partial B^n)$, we consider $f_A:=\min(f(x), A) \in L^2(\partial B^n)$ in the above inequality. Sending $A \to \infty$ , we obtain (via the monotone convergence theorem) the best constant for inequality \eqref{IYW-ball-equv} is given by
$$C(n,\alpha,p,t)=\xi_\alpha=(n\omega_n)^{-\frac{1}{p}}\big(\int_{B^n}\big(\int_{\partial B^n}\frac{1}{|\xi-\zeta|^{n-\alpha}}d\zeta\big)^{t^\prime}d\xi\big)^{\frac{1}{t^\prime}},$$
and the extremal function must be constant function. Sending $p\to (\frac{2(n-1)}{n+\alpha-2})^-, t\to (\frac{2n}{n+\alpha})^-$, we have
$C(n,\alpha,p,t)\to C_{e_1}(n,\alpha)$ for critical powers and constant function is an extremal function for inequality \eqref{IYW-ball-equv}.





To classify all extremal function for inequality \eqref{ineq-HLS-UHS} for $p=\frac{2(n-1)}{n+\alpha-2},t=\frac{2n}{n+\alpha}$,  we define
$$\tilde{E}_\alpha f(x)=\int_{\partial \R^n_+}\frac{f(y)}{|x-y|^{n-\alpha}}dy,~\text{for all}~x\in \R^n_+,$$
where $f\in L^p(\partial \R^n_+)$. Then  sharp inequality \eqref{ineq-HLS-UHS}  is equivalent to
\begin{equation}\label{ineq-HLS-UHS-equv-C}
\|\tilde{E}_\alpha f\|_{L^{t^\prime}(\R^n_+)}\ge C_{e_1}(n,\alpha)\|f\|_{L^p(\partial \R^n_+)}.
\end{equation}
The Euler-Lagrange equation for extremal functions, up to a constant multiplier, is given by
\begin{equation}\label{equ-EL-critical}
f^{p-1}(y)=\int_{\R^n_+}\frac{(\tilde{E}_\alpha f(x))^{t^\prime-1}}{|x-y|^{n-\alpha}}dx, \ \ ~\forall~y\in\partial \R^n_+.
\end{equation}

Let
$u(y)=f^{p-1}(y), v(x)=(\tilde{E}_\alpha f(x))^{t^\prime-1}$. Then
 up to a constant, we have
\begin{equation}\label{equ-UHS-cri}
\begin{cases}
u(y)=\int_{\R^n_+}\frac{v^\kappa(x)}{|x-y|^{n-\alpha}}
dx, ~y\in\partial \R^n_+,\\
v(x)=\int_{\partial \R^n_+}\frac{u^\theta(y)}{|x-y|^{n-\alpha}}dy,~x\in \R^n_+,
\end{cases}
\end{equation}
where $\kappa:=t^\prime-1<0,~ \theta:=\frac{1}{p-1}<0, u(y)\in L^{\theta+1}(\partial \R^n_+),v(x)\in L^{t^\prime}(\R^n_+)$. As in the subcritical case (see Lemma \ref{lm-UHS-regu}),  $u(y), y\in \partial \R^n_+$ is differentiable in $y_i, i=1,2,\cdots,n-1$, and $v(x), x\in \overline{\R^n_+}$ is differentiable in $x_i, i=1,2,\cdots,n$.
Using  the method of moving spheres, (see similar argument given in \cite{DZ-1}-\cite{DZ-2}), we know
$$
u(y)=c_1(\frac{1}{|y-\overline{y}^0|^2+d^2})^{\frac{n-\alpha}{2}},\ \ ~
v(y,0)=c_2(\frac{1}{|y-\overline{y}^0|^2+d^2})^{\frac{n-\alpha}{2}},
$$
where $y,\ \overline{y}^0\in\partial \R^n_+,\ c_1,\ c_2, \ d>0$. Thus
$$
f(y)=c(n,\alpha)(\frac{1}{|y-\overline{y}^0|^2+d^2})^{\frac{n+\alpha-2}{2}},
$$
where $y,\overline{y}^0\in\partial \R^n_+,c(n,\alpha)>0$.

\section{A new  sharp integral inequality on the upper half space} \label{Section 3}
In this section we shall use the similar approach in previous section  to establish the sharp integral inequality \eqref{ineq-HLST-UHS} on the upper half space  \eqref{ineq-HLST-UHS}.  Note that the new kernel can be viewed as the partial derivative of the extension kernel used in Dou and Zhu \cite{DZ-1}. We always assume  $\alpha\ge2$ in this section.


By conformal transformation \eqref{confor-map}, it is easy to see that the inequality \eqref{ineq-HLST-UHS} is equivalent to the following integral inequality on the ball $B^n=B(z^o,1)$ with $z^o=(0,-1)\in \R^{n-1}\times\R$:
\begin{thm}\label{HLST-ball}
For $n\ge3,  2\le\alpha<n$, $p=\frac{2(n-1)}{n+\alpha-4}, t=\frac{2n}{n+\alpha}$, the following sharp inequality holds for all $f\in L^p(\partial B^n), g\in L^t(B^n):$
\begin{equation}\label{ineq-HLST-ball}
\int_{B^n}\int_{\partial B^n}\frac{(1-|x-z^o|^2)g(x)f(y)}{|x-y|^{n-\alpha+2}} dy dx\le C_{e_2}(n,\alpha)\|f\|_{L^p(\partial B^n)}\|g\|_{L^t(B^n)},
\end{equation}
where $C_{e_2}(n,\alpha)$ is given by \eqref{9-19-2}.
\end{thm}

\subsection{A Young Inequality}

We first establish a Young inequality on a ball.

As in previous section we denote  $x^*=z^o+\frac{x-z^o}{|x-z^o|}$ for $x\in B^n\setminus \{z^o\}$ and $(z^o)^*=0$.

\begin{lm}\label{RYT-ball}
Let $p,t>1$ with $1/p+1/t\ge 1$  and $\frac{1}{p}+\frac{1}{p^\prime}=1, \frac{1}{t}+\frac{1}{t^\prime}=1$. Let $h(x, y)$ and $h_1(x,y)$ be two functions defined on $ B^n \times \partial B^n$ and $\partial B^n \times \partial B^n$ respectively. Assume $h(\cdot, x^o)\in L^{at^\prime}(B^n), h_1(0, \cdot)\in L^{(1-a)p^\prime}(\partial B^n)$ for some $0<a<1$, satisfying $|h(x,y)|\le |h_1(x^*,y)|, ~\forall~ x\in B^n, y\in \partial B^n$ and $\|h(\cdot,y^1)\|_{L^{at^\prime}( B^n)}=\|h(\cdot,y^2)\|_{L^{at^\prime}( B^n)}, $ $ ~\forall~ y^1,y^2\in \partial B^n$, $\|h_1(z^1,\cdot)\|_{L^{(1-a)p^\prime}(\partial B^n)}=\|h_1(z^2,\cdot)\|_{L^{(1-a)p^\prime}(\partial B^n)}$, $~\forall~ z^1,z^2\in \partial B^n$. Then for all $f\in L^p(\partial B^n), g\in L^t(B^n)$, there holds
\begin{equation}\label{equ-RYT-ball}
\int_{B^n}\int_{\partial B^n}g(x)f(y)h(x,y)dy dx\le\|f\|_{L^p(\partial B^n)}\|g\|_{L^t(B^n)}\|h(\cdot, x^o)\|_{L^{at^\prime}(B^n)}\|h_1(0,\cdot)\|_{L^{(1-a)p^\prime}(\partial B^n)}.
\end{equation}
\end{lm}

\begin{proof}
Without loss of generality, we can assume that $f,g,h,h_1$ are nonnegative.

Choose $q \ge 1$ such that $\frac{1}{p}+\frac{1}{t}+\frac{1}{q}=2, $ and  $q'=q/(q-1)$ be its conjugate.
Denote
$$I=\int_{B^n}\int_{\partial B^n}g(x)f(y)h(x,y)dy dx$$ and
\begin{eqnarray*}
\gamma_1(x,y)&=&g^{\frac{t}{p^\prime}}(x)h^{1-a}(x,y),\\
\gamma_2(x,y)&=&f^{\frac{p}{t^\prime}}(y)h^a(x,y),\\
\gamma_3(x,y)&=&g^{\frac{t}{q^\prime}}(x)f^{\frac{p}{q^\prime}}(y).
\end{eqnarray*}
(Choose $\gamma_3=1$ if $q'=\infty$).
By H\"{o}lder's inequality, we have
\begin{eqnarray*}
I&\le&\big(\int_{B^n}\int_{\partial B^n}\gamma_3^{q^\prime}(x,y)dy dx\big)^{\frac{1}{q^\prime}}\big(\int_{B^n}\int_{\partial B^n}(\gamma_1(x,y)\gamma_2(x,y))^q dy dx\big)^{\frac{1}{q}}\\
&\le&\big(\int_{B^n}\int_{\partial B^n}\gamma_3^{q^\prime}(x,y)dy dx\big)^{\frac{1}{q^\prime}}\big(\int_{B^n}\int_{\partial B^n}\gamma_1^{p^\prime}(x,y) dy dx\big)^{\frac{1}{p^\prime}}\big(\int_{B^n}\int_{\partial B^n}\gamma_2^{t^\prime}(x,y) dy dx\big)^{\frac{1}{t^\prime}}\\
&=&\|\gamma_3\|_{L^{q^\prime}(B^n\times\partial B^n)}\|\gamma_1\|_{L^{p^\prime}(B^n\times\partial B^n)}\|\gamma_2\|_{L^{t^\prime}(B^n\times\partial B^n)}.
\end{eqnarray*}
On the other hand, it is easy to see
\begin{equation*}
\|\gamma_3\|_{L^{q^\prime}(B^n\times\partial B^n)}=\big(\int_{B^n}\int_{\partial B^n}g^t(x)f^p(y)dy dx\big)^{\frac{1}{q^\prime}}=\|g\|^{\frac{t}{q^\prime}}_{L^t(B^n)}\|f\|^{\frac{p}{q^\prime}}_{L^p(\partial B^n)},
\end{equation*}
and
\begin{eqnarray*}
\|\gamma_2\|_{L^{t^\prime}(B^n\times\partial B^n)}&=&\big(\int_{B^n}\int_{\partial B^n}f^p(y)h^{at^\prime}(x,y)dy dx\big)^{\frac{1}{t^\prime}}\\
&=&\big(\int_{\partial B^n} f^p(y)dy\big)^{\frac{1}{t^\prime}}\big(\int_{B^n} h^{at^\prime}(x,x^o)dx\big)^{\frac{1}{t^\prime}}\\
&=&\|f\|^{\frac{p}{t^\prime}}_{L^p(\partial B^n)}\|h(\cdot,x^o)\|^a_{L^{at^\prime}(B^n)}.
\end{eqnarray*}
Since $h(x,y)\le h_1(x^*,y), ~\forall~ x\in B^n, y\in \partial B^n$ and $\|h_1(z^1,\cdot)\|_{L^{(1-a)p^\prime}(\partial B^n)}=\|h_1(z^2,\cdot)\|_{L^{(1-a)p^\prime}(\partial B^n)}, ~\forall~ z^1,z^2\in \partial B^n$, we have
\begin{eqnarray*}
\|\gamma_1\|_{L^{p^\prime}(B^n\times\partial B^n)}&=&\big(\int_{B^n}\int_{\partial B^n}g^t(x)h^{(1-a)p^\prime}(x,y)dy dx\big)^{\frac{1}{p^\prime}}\\
&\le&\big(\int_{B^n}\int_{\partial B^n}g^t(x)h_1^{(1-a)p^\prime}(x^*,y)dy dx\big)^{\frac{1}{p^\prime}}\\
&=&\big(\int_{B^n} g^t(x)dx\big)^{\frac{1}{p^\prime}}\big(\int_{\partial B^n} h_1^{(1-a)p^\prime}(0,y)dy\big)^{\frac{1}{p^\prime}}\\
&=&\|g\|^{\frac{t}{p^\prime}}_{L^t(B^n)}\|h_1(0, \cdot)\|^{1-a}_{L^{(1-a)p^\prime}(\partial B^n)}.
\end{eqnarray*}
Thus inequality \eqref{equ-RYT-ball} holds.
\end{proof}

\begin{crl}\label{RYWT-ball}
Assume $n\ge3, \ 2\le\alpha<n$, $\frac{1}{p}+\frac{1}{t}\ge1$ and
$$p>\frac{2(n-1)}{n+\alpha-4}, t>\frac{2n}{n+\alpha}.$$
Then there is a constant $C(n,\alpha,p,t)>0$, such that, for all $f\in L^p(\partial B^n), g\in L^t(B^n)$, there holds
\begin{equation}\label{IYWT-ball}
\int_{B^n}\int_{\partial B^n}\frac{(1-|x-z^o|^2)g(x)f(y)}{|x-y|^{n-\alpha+2}} dy dx\le C(n,\alpha,p,t)\|f\|_{L^p(\partial B^n)}\|g\|_{L^t(B^n)}.
\end{equation}
\end{crl}

\begin{proof}
Note that
\begin{eqnarray*}
&&\frac{1-|x-z^o|^2}{|x-y|^{n-\alpha+2}}=\frac{|y-z^o|^2-|x-z^o|^2}{|x-y|^{n-\alpha+2}}=\frac{|x-y|^2-2(y-x,z^o-x)}{|x-y|^{n-\alpha+2}}\\
&\le&\frac{|x-y|+2|z^o-x|}{|x-y|^{n-\alpha+1}}\le \frac{4}{|x-y|^{n-\alpha+1}}.
\end{eqnarray*}
Thus apply Lemma \ref{RYT-ball}  with $$h(x,y)=\frac{4}{|x-y|^{n-\alpha+1}}$$
and
 $$h_1(z,y)=\begin{cases}
1, \quad& \text{if}~|z-y|>\sqrt{2},\\
\frac{1}{(\frac{|z-y|}{2})^{n-\alpha+1}},\quad& \text{if}~|z-y|\le\sqrt{2},
\end{cases}$$
where $x\in B^n, y,z \in\partial B^n$, and $a=\frac{n-\alpha}{2(n-\alpha+1)}$.

It is easy to see that $\|h_1(z^1,\cdot)\|_{L^{(1-a)p^\prime}(\partial B^n)} = \|h_1(z^2,\cdot)\|_{L^{(1-a)p^\prime}(\partial B^n)}, ~\forall~ z^1,z^2\in \partial B^n$, and $h(x,y)\le h_1(x^*,y), ~\forall~ x\in B^n, y\in \partial B^n$.

To use  Lemma \ref{RYT-ball} , we still need to verify
$ h(\cdot, x^o)\in L^{at^\prime}(B^n), h(0, \cdot)\in L^{(1-a)p^\prime}(\partial B^n)$. But these follow from the facts that
$$(n-\alpha+1)at^\prime<n$$
and $$(n-\alpha+1)(1-a)p^\prime<n-1,$$
since $p>\frac{2(n-1)}{n+\alpha-4}, t>\frac{2n}{n+\alpha}$. We conclude the inequality \eqref{IYWT-ball}.
\end{proof}

\begin{rem}\label{rem-UHST}
%
(\rmnum1).
 By conformal transformation \eqref{confor-map}, it is easy to see that inequality \eqref{IYWT-ball} is also equivalent to
\begin{eqnarray}\nonumber
&&\int_{\R^n_+}\int_{\partial \R^n_+}\frac{x_ng(x)f(y)}{(|x^\prime-y|^2+x_n^2)^{\frac{n-\alpha+2}{2}}}\cdot\frac{1}{|(x^\prime,x_n+2)|^{n+\alpha-\frac{2n}{t}}}\cdot\frac{1}{|(y,2)|^{n+\alpha-4-\frac{2(n-1)}{p}}} dy dx\\\label{IYWT-UHST}
&&\le C(n,\alpha,p,t)\|f\|_{L^p(\partial \R^n_+)}\|g\|_{L^t(\R^n_+)}
\end{eqnarray}
for $f\in L^p(\partial \R^n_+), g\in L^t(\R^n_+)$.

(\rmnum2). From the above proof, it is easy to see that  inequality \eqref{IYWT-ball} still holds if  $p\ge\frac{2(n-1)}{n+\alpha-4}, t>\frac{2n}{n+\alpha}$ or $p>\frac{2(n-1)}{n+\alpha-4}, t\ge\frac{2n}{n+\alpha}$ (by taking a different $a\in(0,1)$ in the proof).

\end{rem}

\subsection{Sharp inequality in subcritical case}
Let
$$P_\alpha f(x)=\int_{\partial B^n}\frac{(1-|x-z^o|^2)f(y)}{|x-y|^{n-\alpha+2}}dy,~\forall~x\in B^n,$$
which is obvious continuous for $x\in B^n$ and
$$Q_\alpha g(y)=\int_{B^n}\frac{(1-|x-z^o|^2)g(x)}{|x-y|^{n-\alpha+2}}dx,~{\forall}~y\in \partial B^n.$$
Thus,  inequality \eqref{IYWT-ball} is equivalent to
\begin{equation}\label{IYWT-ball-equv}
\|P_\alpha f\|_{L^{t^\prime}(B^n)}\le C(n,\alpha,p,t)\|f\|_{L^p(\partial B^n)},
\end{equation}
or
\begin{equation}\label{IYWT-ball-equv-add}
\|Q_\alpha g\|_{L^{p^\prime}(\partial B^n)}\le C(n,\alpha,p,t)\|g\|_{L^t(B^n)}.
\end{equation}
The best constant $C_{s_2}(n,\alpha,p,t)$ in above inequalities is given by
\begin{equation} \label{sup-UHST-sub}
C_{s_2}(n,\alpha,p,t)=\sup\{\|P_\alpha f\|_{L^{t^\prime}(B^n)}\big\vert \|f\|_{L^p(\partial B^n)}=1\}.
\end{equation}
We first show that the above supremum can be achieved by some function $f\in L^p(\partial B^n), \|f\|_{L^p(\partial B^n)}=1$.


For simplicity, in the following we just consider $\frac{1}{p}+\frac{1}{t}>1$.
\begin{thm}\label{thm-exist-UHST-sub}
Assume that $2\le \alpha<n$, $p>\frac{2(n-1)}{n+\alpha-4}, t>\frac{2n}{n+\alpha}, \frac{1}{p}+\frac{1}{t}>1$. Then the supremum in \eqref{sup-UHST-sub} is achieved by
a nonnegative $f\in L^p(\partial B^n)$ with $|f\|_{L^p(\partial B^n)}=1$.
\end{thm}

\begin{proof}
Let $\{f_j\}$ be a nonnegative maximizing sequence  to \eqref{sup-UHST-sub} with $\|f_j\|_{L^p(\partial B^n)}=1$.

We first prove that the sequence $\{P_\alpha f_j\}$ is precompact in $L^{t^\prime}(B^n)$.

For fixed $1>r>0$ small, we take
$$\overline{\Omega}_r:=\{x\ : \ |x-z^o|\le 1-r\}.$$
Then for any $x\in \overline{\Omega}_r, y\in \partial B^n,$ we have $|x-y|\ge r$. So for any
$x\in \overline{\Omega}_r$,
$$|P_\alpha f_i(x)|\le r^{n-\alpha+2} \int_{\partial B^n}f(y)dy\le Cr^{n-\alpha+2},$$
which means that $P_\alpha f_i(x)$ is uniformly bounded in $\overline{\Omega}_r$ for any $i$.
On the other hand, for any $i$ and for any $z^1, z^2\in \overline{\Omega}_r$, by using the mean value theorem, there exists $z^\tau:=\tau z^1+(1-\tau )z^2$ for some $0\le \tau \le 1$ such that
\begin{eqnarray*}
|P_\alpha f_i(z^1)-P_\alpha f_i(z^2)|&\le& \int_{\partial B^n}\bigg|\frac{(1-|z^1-z^o|^2)}{|z^1-y|^{n-\alpha+2}}-\frac{(1-|z^2-z^o|^2)}{|z^2-y|^{n-\alpha+2}}\bigg|f(y)dy\\
&\le& C\int_{\partial B^n}\frac{|z^1-z^2|}{|z^\tau-y|^{n-\alpha+3}}f(y)dy\\
&\le& C\frac{1}{r^{n-\alpha+3}}|z^1-z^2|,
\end{eqnarray*}
which implies that $P_\alpha f_i(x)$ is equicontinuous in $\overline{\Omega}_r$. Therefore up to a subsequence there exists $P(x)\in C(\overline{\Omega}_r)$ such that
$$P_\alpha f_i(x)\to P(x)\in C(\overline{\Omega}_r).$$

Now by using \eqref{IYWT-ball-equv} and noticing   Remark \ref{rem-UHST} (\rmnum3), we have
\begin{eqnarray*}
&&\int_{B^n}|P_\alpha f_i(x)-P_\alpha f_j(x)|^{t^\prime}dx\\
&=&\int_{\overline{\Omega}_r}|P_\alpha f_i(x)-P_\alpha f_j(x)|^{t^\prime}dx+\int_{B^n\setminus\overline{\Omega}_r}|P_\alpha f_i(x)-P_\alpha f_j(x)|^{t^\prime}dx\\
&\le&\int_{\overline{\Omega}_r}|P_\alpha f_i(x)-P_\alpha f_j(x)|^{t^\prime}dx\\
&&+(\int_{B^n\setminus\overline{\Omega}_r}|P_\alpha f_i(x)-P_\alpha f_j(x)|^{\frac{2n}{n-\alpha}}dx)^{\frac{t^\prime}{\frac{2n}{n-\alpha}}}\cdot(\int_{B^n\setminus\overline{\Omega}_r}dx)^{1-\frac{t^\prime}{\frac{2n}{n-\alpha}}}\\
&\le&\int_{\overline{\Omega}_r}|P_\alpha f_i(x)-P_\alpha f_j(x)|^{t^\prime}dx+C\|f_i-f_j\|_{L^p(\partial B^n)}r^{1-\frac{t^\prime}{\frac{2n}{n-\alpha}}}\\
&\le&\int_{\overline{\Omega}_r}|P_\alpha f_i(x)-P_\alpha f_j(x)|^{t^\prime}dx+Cr^{1-\frac{t^\prime}{\frac{2n}{n-\alpha}}}.
\end{eqnarray*}
Letting $i,j\to \infty$ and then sending $r\to 0$, we conclude that
$$\int_{B^n}|P_\alpha f_i(x)-P_\alpha f_j(x)|^{t^\prime}dx\to 0,$$
that is, the sequence $\{P_\alpha f_j\}$ is precompact in $L^{t^\prime}(B^n)$.

Since  $\|f_j\|_{L^p(\partial B^n)}=1$,  there exists $f\in L^p(\partial B^n)$ such that $\|f\|_{L^p(\partial B^n)}\le 1$ and $f_j\rightharpoonup f $ weakly in $L^p (\partial B^n)$, up to a subsequence.

Take any  function $g(x)\in L^t(B^n)$, we have $Q_\alpha g(y)\in L^{p^\prime} (\partial B^n)$. Then
$$<P_\alpha f_j-P_\alpha f, g>=<f_j-f, Q_\alpha g> \to 0~\text{as}~i\to\infty,$$
which implies that $P_\alpha f_j-P_\alpha f$ converges weakly to $0$, and then also converges strongly to $0$  in $L^{t^\prime}(B^n)$ since $\{P_\alpha f_j\}$ is precompact in $L^{t^\prime}(B^n)$.  It follows that  $\|f\|_{L^p(\partial B^n)}=1$ and $f$ is a non-negative maximizer of \eqref{sup-UHST-sub}.
\end{proof}

It is standard to show that $f(z)$ is positive, thus satisfies the following  Euler-Lagrange equation:
\begin{equation}\label{equ-HLST-sub}
f^{p-1}(z)=\int_{B^n}\frac{(1-|\xi-z^o|^2)(P_\alpha f(\xi))^{t^\prime-1}}{|\xi-z|^{n-\alpha+2}}d\xi,~\forall~z\in\partial B^n.
\end{equation}
So $\ P_\alpha f(\xi) $ is also   positive.

Let $y=z^{x^o}\in \partial \R^n_+, x=\xi^{x^o}\in \R^n_+$, where $z\in \partial B^n, \xi\in B^n$.  We have


\begin{eqnarray*}
P_\alpha f(x^{x^o})=2\int_{\partial \R^n_+}\frac{x_n f(\zeta^{x^o})}{|x-\zeta|^{n-\alpha+2}}\big(\frac{2}{|x-x^o|}\big)^{-(n-\alpha)}\big(\frac{2}{|\zeta-x^o|}\big)^{n+\alpha-4}d\zeta.
\end{eqnarray*}
Thus \eqref{equ-HLST-sub} becomes
\begin{eqnarray*}
\big(f(y^{x^o})\big)^{p-1}&=&2^{t^\prime}\int_{\R^n_+}\big(\frac{|y-x^o|}2\big)^{n-\alpha+2}\frac{x_n}{|x-y|^{n-\alpha+2}}\big(\frac 2{|x-x^o|}\big)^{n+\alpha}\\
& &\cdot \big(\int_{\partial \R^n_+}\frac{x_n f(\zeta^{x^o})}{|x-\zeta|^{n-\alpha+2}}\big(\frac{2}{|\zeta-x^o|}\big)^{n+\alpha-4}\big(\frac 2{|x-x^o|}\big)^{-(n-\alpha)}d\zeta\big)^{t^\prime-1}dx.
\end{eqnarray*}
Denote $$u(y):=(\frac{2}{|y-x^o|})^{\frac{2(n-1)}{p^\prime}}(f(y^{x^o}))^{p-1}$$ and
$$v(x):=\int_{\partial \R^n_+}\frac{x_n u^{\frac{1}{p-1}}(\zeta)}{|x-\zeta|^{n-\alpha+2}}(\frac{2}{|\zeta-x^o|})^{n+\alpha-4-\frac{2(n-1)}{p}}(\frac{2}{|x-x^o|})^{\frac{2n}{t^\prime}-(n-\alpha)}d\zeta.$$ For simplicity, we also denote $\kappa:=t^\prime-1>0,~ \theta:=\frac{1}{p-1}>0,~ s_1:=n+\alpha-4-\frac{2(n-1)}{p}>0,~  s_2:=n+\alpha-\frac{2n}{t}>0$ since $p>\frac{2(n-1)}{n+\alpha-4}, t>\frac{2n}{n+\alpha}$. Thus single equation \eqref{equ-HLST-sub} can be rewritten as, up to a constant,
$$
\begin{cases}
u(y)=\int_{\R^n_+}\frac{ x_nv^\kappa(x)}{|x-y|^{n-\alpha+2}}
\frac{1}{|y-x^o|^{s_1}}\frac 1{|x-x^o|^{s_2}}dx, ~y\in\partial \R^n_+,\\
v(x)=\int_{\partial \R^n_+}\frac{x_nu^\theta(y)}{|x-y|^{n-\alpha+2}}\frac{1}{|y-x^o|^{s_1}}\frac 1{|x-x^o|^{s_2}}dy,~x\in \R^n_+,
\end{cases}
$$
that is,
\begin{equation}\label{equ2-HLST-sub}
\begin{cases}
u(y)=\int_{\R^n_+}\frac{x_nv^\kappa(x)}{(|x^\prime-y|^2+x_n^2)^{\frac{n-\alpha+2}{2}}}
\frac{1}{|(y,2)|^{s_1}}\frac 1{|(x^\prime,x_n+2)|^{s_2}}dx, ~y\in\partial \R^n_+,\\
v(x)=\int_{\partial \R^n_+}\frac{x_nu^\theta(y)}{(|x^\prime-y|^2+x_n^2)^{\frac{n-\alpha+2}{2}}}\frac{1}{|(y,2)|^{s_1}}\frac 1{|(x^\prime,x_n+2)|^{s_2}}dy,~x\in \R^n_+,
\end{cases}
\end{equation}
where $u(y)\in L^{\theta+1}(\partial \R^n_+)=L^{p'}(\partial \R^n_+),v(x)\in L^{\kappa+1}(\R^n_+)=L^{t^\prime}(\R^n_+)$ and $u(y), v(x)$ are positive functions. Here and in the following we use $y$ to represent points in $\partial \R^n_+$, as well as points in $\R^{n-1}$ .

Next, we shall classify all measurable positive solutions to the system \eqref{equ2-HLST-sub}.

\begin{prop}\label{thm-sym-UHST-sub}
Let $(u,v)$ be a pair of positive solutions to system \eqref{equ2-HLST-sub} with $u(y)\in L^{\theta+1}(\partial \R^n_+),v(x)\in L^{t^\prime}(\R^n_+)$.
For $2\le \alpha<n$, if $p>\frac{2(n-1)}{n+\alpha-4}, t>\frac{2n}{n+\alpha}, \frac{1}{p}+\frac{1}{t}>1$, then $u(y), y\in \partial \R^n_+$ and $v(x), x\in \R^n_+$ must be symmetric with respect to $x_n-$axis.
\end{prop}
Thus we obtain

\begin{prop}\label{thm-sol-UHST-sub}
For $2\le \alpha<n$, if $p>\frac{2(n-1)}{n+\alpha-4}, t>\frac{2n}{n+\alpha}, \frac{1}{p}+\frac{1}{t}>1$, $f\in L^p(\partial B^n)$,  then the positive solution to \eqref{equ-HLST-sub} must be a constant function.
\end{prop}
\begin{proof}
We omit the  proof here since it is the same as that of Proposition \ref{thm-sol-UHS-sub}.
\end{proof}

Now we focus on the proof of Propsition \ref{thm-sym-UHST-sub} via the method of moving planes.

As in Section \ref{Section 2}, we denote $x=\{x_1,x_2,\cdots,x_n\}\in\R^n$ and let $$T_\lambda=\{x\in\R^n \, : \,x_1=\lambda\},x_\lambda=\{2\lambda-x_1,x_2,\cdots,x_n\},x^\prime_\lambda=\{2\lambda-x_1,x_2,\cdots,x_{n-1}\},$$
$$\Sigma_{\lambda,n}=\{x\in\R^n_+ \, : \, x_1\le\lambda\}, \Sigma_{\lambda,n-1}=\{x\in\partial \R^n_+ \ : \ x_1\le\lambda\}.$$
We also denote
$$u_\lambda (y):=u(y_\lambda),~\forall~y\in \partial\R^n_+, \  \ \ \ \
v_\lambda (x):=v(x_\lambda),~\forall~x\in \R^n_+.$$

We need some lemmas.
\begin{lm}\label{lm-UHST-regu}
For $ 2\le \alpha<n$, $p>\frac{2(n-1)}{n+\alpha-4}, t>\frac{2n}{n+\alpha}, \frac{1}{p}+\frac{1}{t}>1$, let $(u,v)$ be a pair of positive solutions to system \eqref{equ2-HLST-sub} with $u(y)\in L^{\theta+1}(\partial \R^n_+),v(x)\in L^{t^\prime}(\R^n_+)$. Then  $u(y), y\in \partial \R^n_+$ and $v(x), x\in \overline{\R^n_+}$ are continuous.
\end{lm}
\begin{proof}
The proof is  standard. So we omit it here. See, for example,   the proofs of  Theorem
$1.3$ in \cite{Li},  Proposition $5.2$ and $5.3$ in \cite{HWY2008}, or Theorem 4.2 in \cite{DZ-2}.
\end{proof}

\begin{lm}\label{lm-UHST-mmp-UV}
Let $(u,v)$ be a pair of  positive solutions to system \eqref{equ2-HLST-sub}. Then we have
\begin{eqnarray}\label{UHST-mmp-U}
&&u(y)-u_\lambda(y)\\\nonumber
&=&\int_{\Sigma_{\lambda,n}}\frac{1}{|(y,2)|^{s_1}}\big(\frac{x_n}{(|x^\prime-y|^2+x_n^2)^{\frac{n-\alpha+2}{2}}}\frac{v^\kappa(x)}{|(x^\prime,x_n+2)|^{s_2}}\\\nonumber
&&+\frac{x_n}{(|x^\prime_\lambda-y|^2+x_n^2)^{\frac{n-\alpha+2}{2}}}\frac{v_\lambda^\kappa(x)}{|(x^\prime_\lambda,x_n+2)|^{s_2}}\big)dx\\\nonumber
&&-\int_{\Sigma_{\lambda,n}}\frac{1}{|(y_\lambda,2)|^{s_1}}\big(\frac{x_n}{(|x^\prime-y|^2+x_n^2)^{\frac{n-\alpha+2}{2}}}\frac{v_\lambda^\kappa(x)}{|(x^\prime_\lambda,x_n+2)|^{s_2}}\\\nonumber
&&+\frac{x_n}{(|x^\prime_\lambda-y|^2+x_n^2)^{\frac{n-\alpha+2}{2}}}\frac{v^\kappa(x)}{|(x^\prime,x_n+2)|^{s_2}}\big)dx, ~y\in\partial \R^n_+,
 \end{eqnarray}
and
\begin{eqnarray}\label{UHST-mmp-V}
&&v(x)-v_\lambda (x)\\\nonumber
&=&\int_{\Sigma_{\lambda,n}}\frac{1}{|(x^\prime,x_n+2)|^{s_2}}\big(\frac{x_n}{(|x^\prime-y|^2+x_n^2)^{\frac{n-\alpha+2}{2}}}\frac{u^\theta(y)}{|(y,2)|^{s_1}}\\\nonumber
&&+\frac{x_n}{(|x^\prime_\lambda-y|^2+x_n^2)^{\frac{n-\alpha+2}{2}}}\frac{u_\lambda^\theta(y)}{|(y_\lambda,2)|^{s_1}}\big)dy\\\nonumber
&&-\int_{\Sigma_{\lambda,n}}\frac{1}{|(x^\prime_\lambda,x_n+2)|^{s_2}}\big(\frac{x_n}{(|x^\prime-y|^2+x_n^2)^{\frac{n-\alpha+2}{2}}}\frac{u_\lambda^\theta(y)}{|(y_\lambda,2)|^{s_1}}\\\nonumber
&&+\frac{x_n}{(|x^\prime_\lambda-y|^2+x_n^2)^{\frac{n-\alpha+2}{2}}}\frac{u^\theta(y)}{|(y,2)|^{s_1}}\big)dy, ~x\in \R^n_+.
\end{eqnarray}
\end{lm}
\begin{proof}
Direct computations give the result. We omit details here.
\end{proof}

For  $t\ge\frac{2n}{n+\alpha}, p>\frac{2(n-1)}{n+\alpha-4}$ or  $t>\frac{2n}{n+\alpha}, p\ge\frac{2(n-1)}{n+\alpha-4}$, inequality \eqref{IYWT-UHST} yields
\begin{eqnarray}\nonumber
&&\big(\int_{\partial \R^n_+}\big(\int_{ \R^n_+}\frac{x_n g(x)}{(|x^\prime-y|^2+x_n^2)^{\frac{n-\alpha+2}{2}}}
\frac{1}{|(x^\prime,x_n+2)|^{s_2}}\frac{1}{|(y,2)|^{s_1}}dx\big)^{p^\prime}dy\big)^{\frac{1}{p^\prime}}\\
\label{IYWT-UHST-1}
&\le& C(n,\alpha,t)\|g\|_{L^t(\R^n_+)},
\end{eqnarray}
and
\begin{eqnarray}\nonumber
&&\big(\int_{\R^n_+}\big(\int_{\partial \R^n_+}\frac{x_n f(y)}{(|x^\prime-y|^2+x_n^2)^{\frac{n-\alpha+2}{2}}}
\frac{1}{|(y,2)|^{s_1}}\frac{1}{|(x^\prime,x_n+2)|^{s_2}}dy\big)^{t^\prime}dx\big)^{\frac{1}{t^\prime}}\\
\label{IYWT-UHST-2}
&\le& C(n,\alpha,p)\|f\|_{L^p(\partial \R^n_+)}.
\end{eqnarray}
\begin{lm}\label{lm-UHST-mmp-start}
Let $(u,v)$ be a pair of  positive solutions to system \eqref{equ2-HLST-sub}  with $u(y)\in L^{\theta+1}(\partial \R^n_+),v(x)\in L^{t^\prime}(\R^n_+)$. Assume that $2\le \alpha<n$, $p>\frac{2(n-1)}{n+\alpha-4}, t>\frac{2n}{n+\alpha}, \frac{1}{p}+\frac{1}{t}>1$.  Then, for sufficiently negative $\lambda$, we have
\begin{eqnarray*}
u(y)&\le & u_\lambda(y),~\forall ~y\in\Sigma_{\lambda,n-1},\\
v(x)&\le & v_\lambda (x),~\forall ~x\in\Sigma_{\lambda,n}.
\end{eqnarray*}
\end{lm}

\begin{proof}
Let $$\Sigma_{\lambda,n}^v=\{x\in\Sigma_{\lambda,n}\ : \ v(x)>v_\lambda (x)\}, \ \mbox{and} \ \ \Sigma_{\lambda,{n-1}}^u=\{y\in\Sigma_{\lambda,{n-1}}\ : \ u(y)>u_\lambda(y)\}.$$

For $\lambda<0$, if $y\in \Sigma_{\lambda,{n-1}}^u$, we have
\begin{eqnarray}\nonumber
0&\le& u(y)-u_\lambda(y)\\\nonumber
&\le&\int_{\Sigma_{\lambda,n}}\frac{1}{|(y,2)|^{s_1}}\big(\frac{x_n}{(|x^\prime-y|^2+x_n^2)^{\frac{n-\alpha+2}{2}}}\frac{v^\kappa(x)}{|(x^\prime,x_n+2)|^{s_2}}\\\nonumber
&&+\frac{x_n}{(|x^\prime_\lambda-y|^2+x_n^2)^{\frac{n-\alpha+2}{2}}}\frac{v_\lambda^\kappa(x)}{|(x^\prime_\lambda,x_n+2)|^{s_2}}\big)dx\\\nonumber
&&-\int_{\Sigma_{\lambda,n}}\frac{1}{|(y,2)|^{s_1}}\big(\frac{x_n}{(|x^\prime-y|^2+x_n^2)^{\frac{n-\alpha+2}{2}}}\frac{v_\lambda^\kappa(x)}{|(x^\prime_\lambda,x_n+2)|^{s_2}}\\\nonumber
&&+\frac{x_n}{(|x^\prime_\lambda-y|^2+x_n^2)^{\frac{n-\alpha+2}{2}}}\frac{v^\kappa(x)}{|(x^\prime,x_n+2)|^{s_2}}\big)dx \\
&=&\int_{\Sigma_{\lambda,n}}\frac{1}{|(y,2)|^{s_1}}\big(\frac{x_n}{(|x^\prime_\lambda-y|^2+x_n^2)^{\frac{n-\alpha+2}{2}}}-\frac{x_n}{(|x^\prime-y|^2+x_n^2)^{\frac{n-\alpha+2}{2}}}\big)\\\nonumber
&&\times\big(\frac{v_\lambda^\kappa(x)}{|(x^\prime_\lambda,x_n+2)|^{s_2}}-\frac{v^\kappa(x)}{|(x^\prime,x_n+2)|^{s_2}}\big)dx\\\nonumber
&\le&\int_{\Sigma_{\lambda,n}^v}\frac{1}{|(y,2)|^{s_1}}\big(\frac{x_n}{(|x^\prime_\lambda-y|^2+x_n^2)^{\frac{n-\alpha+2}{2}}}-\frac{x_n}{(|x^\prime-y|^2+x_n^2)^{\frac{n-\alpha+2}{2}}}\big)\\\nonumber
&&\times\big(\frac{v_\lambda^\kappa(x)}{|(x^\prime_\lambda,x_n+2)|^{s_2}}-\frac{v^\kappa(x)}{|(x^\prime,x_n+2)|^{s_2}}\big)dx\\\label{equ1-UHST-mmp-start}
&\le&\int_{\Sigma_{\lambda,n}^v}\frac{x_n}{(|x^\prime-y|^2+x_n^2)^{\frac{n-\alpha+2}{2}}}\frac{v^\kappa(x)-v_\lambda^\kappa(x)}{|(x^\prime,x_n+2)|^{s_2}}\frac{1}{|(y,2)|^{s_1}}dx,
\end{eqnarray}
where we use the facts that $|x^\prime_\lambda-y|>|x^\prime-y|$,  $|x^\prime|>|x^\prime_\lambda|, \ \mbox{and} \  |y^\prime|>|y^\prime_\lambda|$, for $x\in \Sigma_{\lambda,n}, y\in \Sigma_{\lambda,{n-1}}$  for negative $\lambda$.


Notice that $ \frac{1}{p}+\frac{1}{t}>1$ implies that $\kappa\theta>1$. Now we consider two cases.

Case (\rmnum1). $\kappa>1, \theta>1$.

By \eqref{equ1-UHST-mmp-start}, for $y\in \Sigma_{\lambda,{n-1}}^u$ we have
\begin{eqnarray*}
0&\le& u(y)-u_\lambda(y)\\
&\le&\kappa\int_{\Sigma_{\lambda,n}^v}\frac{x_n}{(|x^\prime-y|^2+x_n^2)^{\frac{n-\alpha+2}{2}}}\frac{\Phi^{\kappa-1}_\lambda(v)(v(x)-v_\lambda (x))}{|(x^\prime,x_n+2)|^{s_2}}\frac{1}{|(y,2)|^{s_1}}dx\\
&\le&\kappa\int_{\Sigma_{\lambda,n}^v}\frac{x_n}{(|x^\prime-y|^2+x_n^2)^{\frac{n-\alpha+2}{2}}}\frac{v^{\kappa-1}(x)(v(x)-v_\lambda (x))}{|(x^\prime,x_n+2)|^{s_2}}\frac{1}{|(y,2)|^{s_1}}dx,
\end{eqnarray*}
where $v_\lambda (x)\le\Phi_\lambda(v)\le v(x)$.  Similarly, for negative $\lambda$ and $x\in \Sigma_{\lambda,n}^u$, we have
\begin{eqnarray*}
0&\le& v(x)-v_\lambda (x)\\
&\le& \theta \int_{\Sigma_{\lambda,n-1}^u}\frac{x_n}{(|x^\prime-y|^2+x_n^2)^{\frac{n-\alpha+2}{2}}}\frac{u^{\theta-1}(y)(u(y)-u_\lambda(y))}{|(y,2)|^{s_1}}\frac{1}{|(x^\prime,x_n+2)|^{s_2}}dy.
\end{eqnarray*}

By \eqref{IYWT-UHST-1}, we have
\begin{eqnarray*}
\|u-u_\lambda\|_{L^{p^\prime}(\Sigma_{\lambda,{n-1}}^u)} 
&\le& C\|v^{\kappa-1}(v-v_\lambda )\|_{L^t(\Sigma_{\lambda,n}^v)}\\ 
 &\le& C\|v\|^{\kappa-1}_{L^{t^\prime}(\Sigma_{\lambda,n}^v)}
\|v-v_\lambda\|_{L^{t^\prime}(\Sigma_{\lambda,n}^v)}.
\end{eqnarray*}
Similarly, \eqref{IYWT-UHST-2} yields  that
\begin{eqnarray}\label{equ2-UHST-mmp-start}
\|v-v_\lambda \|_{L^{t^\prime}(\Sigma_{\lambda,n}^v)}
&\le &C\|u^{\theta-1}(u-u_\lambda)\|_{L^p(\Sigma_{\lambda,{n-1}}^u)}\nonumber\\
 &\le& C\|u\|^{\theta-1}_{L^{\theta+1}(\Sigma_{\lambda,{n-1}}^u)}\|u-u_\lambda\|_{L^{p^\prime}(\Sigma_{\lambda,{n-1}}^u)}.
\end{eqnarray}

Notice that  $u(y)\in L^{\theta+1}(\partial \R^n_+),v(x)\in L^{t^\prime}(\R^n_+)$. Thus  for sufficiently negative $\lambda$, $\|v\|^{\kappa-1}_{L^{t^\prime}(\Sigma_{\lambda,n}^v)}, \|u\|^{\theta-1}_{L^{\theta+1}(\Sigma_{\lambda,{n-1}}^u)}$ will be close to zero, which yields  that $\|v-v_\lambda \|_{L^{t^\prime}(\Sigma_{\lambda,n}^v)}=\|u-u_\lambda\|_{L^{p^\prime}(\Sigma_{\lambda,{n-1}}^u)}=0$. That is: both $\Sigma_{\lambda,n}^v$ and $\Sigma_{\lambda,{n-1}}^u$ are empty.

Case (\rmnum2). $0<\kappa\le 1, \theta>1$ or $\kappa>1, 0<\theta\le 1$.

We assume without loss of generality that $0<\kappa\le 1, \theta>1$.  Since $\frac1p+\frac1t>1,$ we have $\frac1\kappa<\theta$. Thus there exists a number $r$ satisfied $\frac1\kappa <r<\theta$. By \eqref{equ1-UHST-mmp-start},  for $\lambda<0, \  y\in \Sigma_{\lambda,{n-1}}^u$, we have
\begin{eqnarray*}
0\le u(y)-u_\lambda(y)&\le& \int_{\Sigma_{\lambda,n}^v}\frac{x_n}{(|x^\prime-y|^2+x_n^2)^{\frac{n-\alpha+2}{2}}}
\frac{ v^\kappa(x)-v^\kappa_\lambda(x) }{|(x^\prime,x_n+2)|^{s_2}}\frac{1}{|(y,2)|^{s_1}}dx\\
&\le& \kappa r\int_{\Sigma_{\lambda,n}^v}\frac{x_n}{(|x^\prime-y|^2+x_n^2)^{\frac{n-\alpha+2}{2}}}
\frac{v^{\kappa-\frac1r}(v^\frac1r(x)-v^\frac1r_\lambda(x))}{|(x^\prime,x_n+2)|^{s_2}}\frac{1}{|(y,2)|^{s_1}}dx.
\end{eqnarray*}
By \eqref{IYWT-UHST-1}, we have
\begin{eqnarray}\label{equ2-UHST-mmp-start-1}
 \|u-u_\lambda\|_{L^{p^\prime}(\Sigma_{\lambda,{n-1}}^u)} 
&\le&C\|v^{\kappa-\frac1r}(v^\frac1r-v^\frac1r_\lambda )\|_{L^t(\Sigma_{\lambda,n}^v)}\nonumber\\
&\le&C\|v^{\kappa-\frac1r}\|_{L^{t^\prime/(t^\prime-1-\frac1r)}(\Sigma_{\lambda,n}^v)}
\|v^\frac1r-v^\frac1r_\lambda\|_{L^{t^\prime r}(\Sigma_{\lambda,n}^v)}\nonumber\\
&=&C\|v\|^{\kappa-\frac1r}_{L^{t^\prime}(\Sigma_{\lambda,n}^v)}
\|v^\frac1r-v^\frac1r_\lambda\|_{L^{t^\prime r}(\Sigma_{\lambda,n}^v)}.
\end{eqnarray}
Using the facts that $|x^\prime|>|x^\prime_\lambda|, \ \mbox{and} \  |y^\prime|>|y^\prime_\lambda|$, for $x\in \Sigma_{\lambda,n}, y\in \Sigma_{\lambda,{n-1}}$ with $\lambda<0$, we have\\
\begin{eqnarray*}
v(x)&=&\int_{\Sigma_{\lambda,{n-1}}^u}\frac{x_nu^\theta(y)}{(|x^\prime-y|^2+x_n^2)^{\frac{n-\alpha+2}{2}}}\frac{1}{|(y,2)|^{s_1}}\frac 1{|(x^\prime,x_n+2)|^{s_2}}dy\\
& &+\int_{\Sigma_{\lambda,{n-1}}^u}\frac{x_nu_\lambda^\theta(y)}{(|x^\prime_\lambda-y|^2+x_n^2)^{\frac{n-\alpha+2}{2}}}\frac{1}{|(y_\lambda,2)|^{s_1}}\frac 1{|(x^\prime,x_n+2)|^{s_2}}dy\\
& &+\int_{\Sigma_{\lambda,{n-1}}\setminus\Sigma_{\lambda,{n-1}}^u}
\frac{x_nu^\theta(y)}{(|x^\prime-y|^2+x_n^2)^{\frac{n-\alpha+2}{2}}}\frac{1}{|(y,2)|^{s_1}}\frac 1{|(x^\prime,x_n+2)|^{s_2}}dy\\
& &+\int_{\Sigma_{\lambda,{n-1}}\backslash\Sigma_{\lambda,{n-1}}^u}
\frac{x_nu_\lambda^\theta(y)}{(|x^\prime_\lambda-y|^2+x_n^2)^{\frac{n-\alpha+2}{2}}}\frac{1}{|(y_\lambda,2)|^{s_1}}\frac 1{|(x^\prime,x_n+2)|^{s_2}}dy\\
&\le&\int_{\Sigma_{\lambda,{n-1}}^u}\frac{x_nu^\theta(y)}{(|x^\prime-y|^2+x_n^2)^{\frac{n-\alpha+2}{2}}}\frac{1}{|(y,2)|^{s_1}}\frac 1{|(x^\prime_\lambda,x_n+2)|^{s_2}}dy\\
& &+\int_{\Sigma_{\lambda,{n-1}}^u}\frac{x_nu_\lambda^\theta(y)}{(|x^\prime_\lambda-y|^2+x_n^2)^{\frac{n-\alpha+2}{2}}}\frac{1}{|(y,2)|^{s_1}}\frac 1{|(x^\prime_\lambda,x_n+2)|^{s_2}}dy\\
& &+\int_{\Sigma_{\lambda,{n-1}}^u}\frac{x_nu_\lambda^\theta(y)}{(|x^\prime-y|^2+x_n^2)^{\frac{n-\alpha+2}{2}}}\frac 1{|(x^\prime_\lambda,x_n+2)|^{s_2}}\big(\frac{1}{|(y_\lambda,2)|^{s_1}}-\frac{1}{|(y,2)|^{s_1}}\big)dy\\
& &+\int_{\Sigma_{\lambda,{n-1}}\setminus\Sigma_{\lambda,{n-1}}^u}
\frac{x_nu^\theta(y)}{(|x^\prime-y|^2+x_n^2)^{\frac{n-\alpha+2}{2}}}\frac{1}{|(y,2)|^{s_1}}\frac 1{|(x^\prime,x_n+2)|^{s_2}}dy\\
& &+\int_{\Sigma_{\lambda,{n-1}}\backslash\Sigma_{\lambda,{n-1}}^u}
\frac{x_nu_\lambda^\theta(y)}{(|x^\prime_\lambda-y|^2+x_n^2)^{\frac{n-\alpha+2}{2}}}\frac{1}{|(y_\lambda,2)|^{s_1}}\frac 1{|(x^\prime,x_n+2)|^{s_2}}dy,
\end{eqnarray*}and
 \begin{eqnarray*}
v_\lambda(x)&=&\int_{\Sigma_{\lambda,{n-1}}^u}\frac{x_nu^\theta(y)}{(|x^\prime_\lambda-y|^2+x_n^2)^{\frac{n-\alpha+2}{2}}}\frac{1}{|(y,2)|^{s_1}}\frac 1{|(x^\prime_\lambda,x_n+2)|^{s_2}}dy\\
& &+\int_{\Sigma_{\lambda,{n-1}}^u}\frac{x_nu_\lambda^\theta(y)}{(|x^\prime_\lambda-y_\lambda|^2+x_n^2)^{\frac{n-\alpha+2}{2}}}
\frac{1}{|(y_\lambda,2)|^{s_1}}\frac 1{|(x^\prime_\lambda,x_n+2)|^{s_2}}dy\\
& &+\int_{\Sigma_{\lambda,{n-1}}\setminus\Sigma_{\lambda,{n-1}}^u}
\frac{x_nu^\theta(y)}{(|x^\prime_\lambda-y|^2+x_n^2)^{\frac{n-\alpha+2}{2}}}\frac{1}{|(y,2)|^{s_1}}\frac 1{|(x^\prime_\lambda,x_n+2)|^{s_2}}dy\\
& &+\int_{\Sigma_{\lambda,{n-1}}\backslash\Sigma_{\lambda,{n-1}}^u}
\frac{x_nu_\lambda^\theta(y)}{(|x^\prime_\lambda-y_\lambda|^2+x_n^2)^{\frac{n-\alpha+2}{2}}}\frac{1}{|(y_\lambda,2)|^{s_1}}\frac 1{|(x^\prime_\lambda,x_n+2)|^{s_2}}dy\\
&\ge&\int_{\Sigma_{\lambda,{n-1}}^u}\frac{x_nu^\theta(y)}{(|x^\prime_\lambda-y|^2+x_n^2)^{\frac{n-\alpha+2}{2}}}\frac{1}{|(y,2)|^{s_1}}\frac 1{|(x^\prime_\lambda,x_n+2)|^{s_2}}dy\\
& &+\int_{\Sigma_{\lambda,{n-1}}^u}\frac{x_nu_\lambda^\theta(y)}{(|x^\prime-y|^2+x_n^2)^{\frac{n-\alpha+2}{2}}}
\frac{1}{|(y,2)|^{s_1}}\frac 1{|(x^\prime_\lambda,x_n+2)|^{s_2}}dy\\
& &+\int_{\Sigma_{\lambda,{n-1}}^u}\frac{x_nu_\lambda^\theta(y)}{(|x^\prime-y|^2+x_n^2)^{\frac{n-\alpha+2}{2}}}\frac 1{|(x^\prime_\lambda,x_n+2)|^{s_2}}\big(\frac{1}{|(y_\lambda,2)|^{s_1}}-\frac{1}{|(y,2)|^{s_1}}\big)dy\\
& &+\int_{\Sigma_{\lambda,{n-1}}\setminus\Sigma_{\lambda,{n-1}}^u}
\frac{x_nu^\theta(y)}{(|x^\prime-y|^2+x_n^2)^{\frac{n-\alpha+2}{2}}}\frac{1}{|(y,2)|^{s_1}}\frac 1{|(x^\prime,x_n+2)|^{s_2}}dy\\
\end{eqnarray*}
 \begin{eqnarray*}
& &+\int_{\Sigma_{\lambda,{n-1}}\backslash\Sigma_{\lambda,{n-1}}^u}
\frac{x_nu_\lambda^\theta(y)}{(|x^\prime_\lambda-y|^2+x_n^2)^{\frac{n-\alpha+2}{2}}}\frac{1}{|(y_\lambda,2)|^{s_1}}\frac 1{|(x^\prime,x_n+2)|^{s_2}}dy.
\end{eqnarray*}
Noticing: for 
$0<s\le1,\, a\ge  b\ge0$ and $c\ge0,$
$$(a+c)^s-(b+c)^s\le a^s-b^s,$$
and
the fact that $|x^\prime_\lambda-y|>|x^\prime-y|$,  for $x\in \Sigma_{\lambda,n}, y\in \Sigma_{\lambda,{n-1}}$ with $\lambda<0$,
 we have, for $x\in\Sigma_{\lambda,n}^v$,
\begin{eqnarray*}
0&\le&v^\frac1r-v_\lambda^\frac1r\\
&\le&\big(\int_{\Sigma_{\lambda,{n-1}}^u}\frac{x_nu^\theta(y)}{(|x^\prime-y|^2+x_n^2)^{\frac{n-\alpha+2}{2}}}\frac{1}{|(y,2)|^{s_1}}\frac 1{|(x^\prime_\lambda,x_n+2)|^{s_2}}dy\\
& &+\int_{\Sigma_{\lambda,{n-1}}^u}\frac{x_nu_\lambda^\theta(y)}{(|x^\prime_\lambda-y|^2+x_n^2)^{\frac{n-\alpha+2}{2}}}\frac{1}{|(y,2)|^{s_1}}\frac 1{|(x^\prime_\lambda,x_n+2)|^{s_2}}dy\big)^\frac1r\\
& &-\big(\int_{\Sigma_{\lambda,{n-1}}^u}\frac{x_nu^\theta(y)}{(|x^\prime_\lambda-y|^2+x_n^2)^{\frac{n-\alpha+2}{2}}}\frac{1}{|(y,2)|^{s_1}}\frac 1{|(x^\prime_\lambda,x_n+2)|^{s_2}}dy\\
& &+\int_{\Sigma_{\lambda,{n-1}}^u}\frac{x_nu_\lambda^\theta(y)}{(|x^\prime-y|^2+x_n^2)^{\frac{n-\alpha+2}{2}}}
\frac{1}{|(y,2)|^{s_1}}\frac 1{|(x^\prime_\lambda,x_n+2)|^{s_2}}dy\big)^\frac1r\\
&=&\big[\int_{\Sigma_{\lambda,{n-1}}^u}\big(\frac{(x_nu^\theta(y))^\frac1r}{(|x^\prime-y|^2+x_n^2)^{\frac{n-\alpha+2}{2r}}}\frac{1}{|(y,2)|^\frac{s_1}r}\frac 1{|(x^\prime_\lambda,x_n+2)|^\frac{s_2}r}\big)^r\\
& &+\big(\frac{(x_nu_\lambda^\theta(y))^\frac1r}{(|x^\prime_\lambda-y|^2+x_n^2)^{\frac{n-\alpha+2}{2r}}}\frac{1}{|(y,2)|^\frac{s_1}r}\frac 1{|(x^\prime_\lambda,x_n+2)|^\frac{s_2}r}\big)^rdy\big]^\frac1r\\
& &-\big[\int_{\Sigma_{\lambda,{n-1}}^u}\big(\frac{(x_nu^\theta(y))^\frac1r}{(|x^\prime_\lambda-y|^2+x_n^2)^{\frac{n-\alpha+2}{2r}}}
\frac{1}{|(y,2)|^\frac{s_1}r}\frac 1{|(x^\prime_\lambda,x_n+2)|^\frac{s_2}r}\big)^r\\
& &+\big(\frac{(x_nu_\lambda^\theta(y))^\frac1r}{(|x^\prime-y|^2+x_n^2)^{\frac{n-\alpha+2}{2r}}}\frac{1}{|(y,2)|^\frac{s_1}r}\frac 1{|(x^\prime_\lambda,x_n+2)|^\frac{s_2}r}\big)^rdy\big]^\frac1r\\
&\le&\big[\int_{\Sigma_{\lambda,{n-1}}^u}\big(\frac{x_n\big(u^\frac\theta r(y)- u_\lambda^\frac\theta r(y)\big)^r}{(|x^\prime-y|^2+x_n^2)^{\frac{n-\alpha+2}{2}}}\frac1{|(y,2)|^{s_1}}\frac 1{|(x^\prime_\lambda,x_n+2)|^{s_2}}\\
& &+\frac{x_n\big(u^\frac\theta r(y)- u_\lambda^\frac\theta r(y)\big)^r}{(|x^\prime_\lambda-y|^2+x_n^2)^{\frac{n-\alpha+2}{2}}}\frac1{|(y,2)|^{s_1}}\frac 1{|(x^\prime_\lambda,x_n+2)|^{s_2}}\big)dy\big]^\frac1r\\
&\le&2\big[\int_{\Sigma_{\lambda,{n-1}}^u}\frac{x_n(u^\frac\theta r(y)-u_\lambda^\frac\theta r(y))^r}{(|x^\prime-y|^2+x_n^2)^{\frac{n-\alpha+2}{2}}}\frac{1}{|(y,2)|^{s_1}}\frac 1{|(x^\prime_\lambda,x_n+2)|^{s_2}}dy\big]^\frac1r\\
&\le&\frac{2\theta} r\big[\int_{\Sigma_{\lambda,{n-1}}^u}\frac{x_nu^{\theta-r}(y)(u(y)-u_\lambda(y))^r}{(|x^\prime-y|^2+x_n^2)^{\frac{n-\alpha+2}{2}}}
\frac{1}{|(y,2)|^{s_1}}\frac 1{|(x^\prime_\lambda,x_n+2)|^{s_2}}dy\big]^\frac1r.\\
\end{eqnarray*}
Then, by HLS inequality \eqref{IYWT-UHST-1} it yields
\begin{eqnarray}\label{equ2-UHST-mmp-start-2}
& &\|v^\frac1r-v^\frac1r_\lambda\|_{L^{t^\prime r}(\Sigma_{\lambda,n}^v)}\nonumber\\
&\le &\frac{2\theta} r\big\{\int_{\Sigma_{\lambda,n}^v}\big[\int_{\Sigma_{\lambda,{n-1}}^u}\frac{x_nu^{\theta-r}(y)(u(y)-u_\lambda(y))^r}{(|x^\prime-y|^2+x_n^2)^{\frac{n-\alpha+2}{2}}}
\frac{1}{|(y,2)|^{s_1}}\frac 1{|(x^\prime_\lambda,x_n+2)|^{s_2}}dy\big]^{t^\prime}dx\big\}^\frac1{t^\prime r}\nonumber\\
&\le&C\|u^{\theta-r}(u-u_\lambda)^r\|^\frac1r_{L^{p}(\Sigma_{\lambda,{n-1}}^u)}\nonumber\\
&\le&C\|u\|^\frac{\theta-r}r_{L^{\theta+1}(\Sigma_{\lambda,{n-1}}^u)}\|u-u_\lambda\|_{L^{p^\prime}(\Sigma_{\lambda,{n-1}}^u)}.
\end{eqnarray}
Combining the above with \eqref{equ2-UHST-mmp-start-1}, we have
\begin{eqnarray*}
\|u-u_\lambda\|_{L^{p^\prime}(\Sigma_{\lambda,{n-1}}^u)}&\le& C\|v\|^{\kappa-\frac1r}_{L^{t^\prime}(\Sigma_{\lambda,n}^v)}
\|v^\frac1r-v^\frac1r_\lambda\|_{L^{t^\prime r}(\Sigma_{\lambda,n}^v)}\\
&\le&C\|v\|^{\kappa-\frac1r}_{L^{t^\prime}(\Sigma_{\lambda,n}^v)}\|u\|^\frac{\theta-r}r_{L^{\theta+1}(\Sigma_{\lambda,{n-1}}^u)}\|u-u_\lambda\|_{L^{p^\prime}(\Sigma_{\lambda,{n-1}}^u)}.
\end{eqnarray*}
Then, we can obtain \eqref{equ2-UHST-mmp-start} again. Similar to that in Case (\rmnum1) we can show that both $\Sigma_{\lambda,{n-1}}^u$ and $\Sigma_{\lambda,n}^v$are empty.
\end{proof}

Define $$\overline{\lambda}=\sup\{\mu<0\ : \forall \lambda\le \mu, \  \ u(y)\le u_\lambda(y),~\forall ~y\in\Sigma_{\lambda,n-1}; \ v(x)\le v_\lambda (x),~\forall ~x\in\Sigma_{\lambda,n}\}.$$

\begin{lm}\label{lm-UHST-mmp-infy}
$\overline{\lambda}=0$.
\end{lm}
\begin{proof}
We prove this  by contradiction. Suppose that $\overline{\lambda}<0$.

Notice that
\begin{eqnarray*}
u(y)&\le & u_{\overline{\lambda}}(y),~\forall ~y\in\Sigma_{\overline{\lambda},n-1},\\
v(x)&\le & v_{\overline{\lambda}}(x),~\forall ~x\in\Sigma_{\overline{\lambda},n}.
\end{eqnarray*}
Then by Lemma \ref{lm-UHST-mmp-UV} we have
\begin{eqnarray}\label{equ1-UHST-mmp-infy}
u(y)&< & u_{\overline{\lambda}}(y),~\forall ~y\in\Sigma_{\overline{\lambda},n-1},\\\label{equ2-UHST-mmp-infy}
v(x)&< & v_{\overline{\lambda}}(x),~\forall ~x\in\Sigma_{\overline{\lambda},n}.
\end{eqnarray}
Otherwise, if there exists $y^0\in\Sigma_{\overline{\lambda},n-1}$ such that $u(y^0)=u(y^0_{\overline{\lambda}})$, then
\begin{eqnarray*}
0&\le&\int_{\Sigma_{\overline{\lambda},n}}\frac{1}{|(y^0,2)|^{s_1}}\big(\frac{x_n}{(|x^\prime-y^0|^2+x_n^2)^{\frac{n-\alpha+2}{2}}}
-\frac{x_n}{(|x^\prime_{\overline{\lambda}}-y^0|^2+x_n^2)^{\frac{n-\alpha+2}{2}}}\big)\\
&&\times\big(\frac{v^\kappa(x)}{|(x^\prime,x_n+2)|^{s_2}}-\frac{v_{\overline{\lambda}}^\kappa(x)}{|(x^\prime_{\overline{\lambda}},x_n+2)|^{s_2}}\big),
 \end{eqnarray*}
which yields
$$v(x)\equiv v_{\overline{\lambda}}(x)\equiv 0,~\forall ~ x\in \R^n_+$$
since $|x_{\overline{\lambda}}-y^0|>|x-y^0|, |x_{\overline{\lambda}}|<|x|$.
It is impossible.

Now we prove that, there exists $\epsilon^*>0$ such that, for any $\lambda\in(\overline{\lambda},\overline{\lambda}+\epsilon^*)$,
\begin{eqnarray*}
u(y)&\le & u_\lambda(y),~\forall ~y\in\Sigma_{\lambda,n-1},\\
v(x)&\le & v_\lambda (x),~\forall ~x\in\Sigma_{\lambda,n},
\end{eqnarray*}
which draws a contradiction, we thus have $\overline{\lambda}=0$.


By Lemma \ref{lm-UHST-regu}, $u,v$ are continuous. Then by \eqref{equ1-UHST-mmp-infy}, \eqref{equ2-UHST-mmp-infy}, for $R>0$ large we can take $\delta>0$ small such that
\begin{eqnarray*}
&&u_{\overline{\lambda}}(y)-u(y)\ge C_0,~\forall ~y\in\Sigma_{\overline{\lambda}-\delta,n-1}\cap B^{n-1}_R(0),\\
&&v_{\overline{\lambda}}(x)-v(x)\ge C_0,~\forall ~x\in\Sigma_{\overline{\lambda}-\delta,n}\cap B^n_R(0)
\end{eqnarray*}
for some $C_0>0$ small.
Then there exists $\epsilon_1>0$ small such that for  $\lambda\in(\overline{\lambda},\overline{\lambda}+\epsilon_1)$,
\begin{eqnarray*}\label{equ3-UHST-mmp-infy}
&&u_{\lambda}(y)-u(y)\ge 0,~\forall ~y\in\Sigma_{\overline{\lambda}-\delta,n-1}\cap B^{n-1}_R(0),\\\label{equ4-UHST-mmp-infy}
&&v_{\lambda}(x)-v(x)\ge 0,~\forall ~x\in\Sigma_{\overline{\lambda}-\delta,n}\cap B^n_R(0).
\end{eqnarray*}

So  $\Sigma_{\lambda,n-1}^u\subset (\partial \R^n_+ \setminus B^{n-1}_R(0))\cup((\Sigma_{\lambda,n-1}\setminus \Sigma_{\overline{\lambda}-\delta,n-1})\cap B^{n-1}_R(0))$ and $\Sigma_{\lambda,n}^v\subset (\R^n_+\setminus B^n_R(0)) \cup ((\Sigma_{\lambda,n}\setminus \Sigma_{\overline{\lambda}-\delta,n})\cap B^n_R(0))$.
Therefore we can first choose $R>0$ large, then $\delta>0$ small, then $0<\epsilon_2<\epsilon_1$ small such that for $\lambda\in(\overline{\lambda},\overline{\lambda}+\epsilon_2)$,
$$\|v\|_{L^{t^\prime}((\R^n_+\setminus B^n_R(0)) \cup ((\Sigma_{\lambda,n}\setminus \Sigma_{\overline{\lambda}-\delta,n})\cap B^n_R(0)))}$$ and $$\|u\|_{L^{\theta+1}((\partial \R^n_+ \setminus B^{n-1}_R(0))\cup((\Sigma_{\lambda,n-1}\setminus \Sigma_{\overline{\lambda}-\delta,n-1})\cap B^{n-1}_R(0)))}$$ is sufficiently small. Then the same argument as in Lemma \ref{lm-UHST-mmp-start} yields  that both $\Sigma_{\lambda,n}^v$ and $\Sigma_{\lambda,{n-1}}^u$ are empty.
%
\end{proof}

{\bf Proof of Propsition \ref{thm-sym-UHST-sub}.}
Lemma \ref{lm-UHST-mmp-infy} indicates $\overline{\lambda}=0$. Thus $u(y)$ and $v(x)$ are symmetric with respect to $x_1=0$ since one also can move the plane from positive infinity to zero.
Similar argument shows that  $u(y)$ and $v(x)$ are symmetric with respect to $x_2=0, x_3=0, \cdots,x_{n-1}=0$. \hfill$\Box$

\subsection{Proof of Theorem \ref{HLST-ball}} By Proposition \ref{thm-sol-UHST-sub}, the best constant for inequality \eqref{IYWT-ball-equv} is
$$C(n,\alpha,p,t)=(n\omega_n)^{-\frac{1}{p}}\big(\int_{B^n}\big(\int_{\partial B^n}\frac{(1-|\xi-z^o|^2)}{|\xi-\zeta|^{n-\alpha+2}}d\zeta\big)^{t^\prime}d\xi\big)^{\frac{1}{t^\prime}}.$$
Sending $p\to (\frac{2(n-1)}{n+\alpha-4})^+, t\to (\frac{2n}{n+\alpha})^+$, we have $C(n,\alpha,p,t)\to C_{e_2}(n,\alpha)$, where
$$C_{e_2}(n,\alpha) =(n\omega_n)^{-\frac{n+\alpha-4}{2(n-1)}}\big(\int_{B^n}\big(\int_{\partial B^n}\frac{(1-|\xi-z^o|^2)}{|\xi-\zeta|^{n-\alpha+2}}d\zeta\big)^{\frac{2n}{n-\alpha}}d\xi\big)^{\frac{n-\alpha}{2n}}.$$

%
%

We are left to classify all extremal functions.
In the rest of this subsection we always assume $p=\frac{2(n-1)}{n+\alpha-4}, t=\frac{2n}{n+\alpha}$.

 For $2\le\alpha<n$, define
$$\tilde{P}_\alpha f(x)=\int_{\partial \R^n_+}\frac{x_n f(y)}{(|x^\prime-y|^2+x_n^2)^{\frac{n-\alpha+2}{2}}}dy,~\text{for all}~x\in \R^n_+,$$
where $f\in L^p(\partial \R^n_+)$. Then inequality \eqref{ineq-HLST-UHS} is equivalent to
\begin{equation}\label{ineq-HLST-UHS-equv-C}
\|\tilde{P}_\alpha f\|_{L^{t^\prime}(\R^n_+)}\le \frac{C_{e_2}(n,\alpha)}{2}\|f\|_{L^p(\partial \R^n_+)}.
\end{equation}
The Euler-Lagrange equation to the above inequality is
\begin{equation}\label{equ-EL-critical-HLST}
f^{p-1}(y)=\int_{\R^n_+}\frac{x_n(\tilde{P}_\alpha f(x))^{t^\prime-1}}{(|x^\prime-y|^2+x_n^2)^{\frac{n-\alpha+2}{2}}}dx, \ \ ~\forall~y\in\partial \R^n_+.
\end{equation}


\begin{prop}\label{prop-extr-HLST}
Let $f\in L^p(\partial \R^n_+)$ be a positive solution to equation \eqref{equ-EL-critical-HLST}. Then $f$ must be given in the following form:
\begin{equation*}
f(y)=c(n,\alpha)(\frac{1}{|y-\overline{y}^0|^2+d^2})^{\frac{n+\alpha-4}{2}},
\end{equation*}
where $y, \overline{y}^0\in\partial \R^n_+, c(n,\alpha), d>0$.
\end{prop}

To prove Proposition \ref{prop-extr-HLST}, for simplicity we again denote $u(y)=f^{p-1}(y), v(x)=(\tilde{P}_\alpha f(x))^{t^\prime-1}$, then
 up to a constant, we have
\begin{equation}\label{equ-UHST-cri}
\begin{cases}
u(y)=\int_{\R^n_+}\frac{x_n v^\kappa(x)}{|x-y|^{n-\alpha+2}}
dx, ~y\in\partial \R^n_+,\\
v(x)=\int_{\partial \R^n_+}\frac{x_n u^\theta(y)}{|x-y|^{n-\alpha+2}}dy,~x\in \R^n_+,
\end{cases}
\end{equation}
where $\kappa:=t^\prime-1>0,~ \theta:=\frac{1}{p-1}>0, u(y)\in L^{\theta+1}(\partial \R^n_+),v(x)\in L^{t^\prime}(\R^n_+)$. As in the subcritical case (see Lemma \ref{lm-UHST-regu}),  $u(y), y\in \partial \R^n_+$ and $v(x), x\in \overline{\R^n_+}$ are continuous.

We shall use the method of moving spheres to derive
\begin{prop}\label{prop-extr-sys-HLST}
Let $(u,v)$ be a pair of positive solutions to system \eqref{equ-UHST-cri} with $u(y)\in L^{\theta+1}(\partial \R^n_+),v(x)\in L^{t^\prime}(\R^n_+)$. Then $u,v$ must be the following forms:
\begin{eqnarray*}
u(y)&=&c_1\big(\frac{1}{|y-\overline{y}^0|^2+d^2}\big)^{\frac{n-\alpha+2}{2}},\\
v(y,0)&=&c_2\big(\frac{1}{|y-\overline{y}^0|^2+d^2}\big)^{\frac{n-\alpha}{2}},
\end{eqnarray*}
where $y,\overline{y}^0\in\partial \R^n_+,c_1,c_2,d>0$.
\end{prop}
Proposition \eqref {prop-extr-HLST} follows from the above proposition immediately.

To use the method of moving spheres, we introduce the following notations:
for $R>0, y^0\in\partial \R^n_+$, denote
$$B_R(y^0)=\{z\in\R^n\,: \,|z-y^0|<R\}, $$
$$B_R^{n-1}(y^0)=\{z\in\partial\R^{n}_+\,: \,|z-y^0|<R\},$$
$$B_R^+(y^0)=\{z=(z_1,z_2,\cdots,z_n)\in B_R(y^0)\,: \, z_n>0\},$$
and
$$\Sigma_{y^0,R}^n=\R^n_+\setminus \overline{B_R^+(y^0)}, \Sigma_{y^0,R}^{n-1}=\partial \R^n_+\setminus \overline{B_R^{n-1}(y^0)}.$$
For $y^0\in\partial \R^n_+, \lambda>0$, we define
$$u_{y^0, \lambda} (\eta):=\big(\frac{\lambda}{|\eta-y^0|}\big)^{n-\alpha+2}u(\eta^{y^0,\lambda}),~\mbox{for}~\eta \in \partial\R^n_+,$$
$$v_{y^0, \lambda} (\xi):=\big(\frac{\lambda}{|\xi-y^0|}\big)^{n-\alpha}v(\xi^{y^0,\lambda}),~\mbox{for}~\xi\in \R^n_+,$$
where for any $z\in \overline{\R^n_+}$,
$$z^{y^0,\lambda}:=\frac{\lambda^2(z-y^0)}{|z-y^0|^2}+y^o.$$
\begin{lm}\label{lm-UHST-mms-UV-cri}
Let $(u,v)$ be a pair of positive solutions to system \eqref{equ-UHST-cri}. We have
\begin{eqnarray*}
u_{y^0, \lambda} (\eta)=\int_{\R^n_+}\frac{\xi_n v_{y^0, \lambda}^\kappa (\xi)}{|\xi-\eta|^{n-\alpha+2}}
d\xi, ~\eta\in\partial \R^n_+,\\
v_{y^0, \lambda} (\xi)=\int_{\partial \R^n_+}\frac{\xi_n u_{y^0, \lambda}^\theta (\eta)}{|\xi-\eta|^{n-\alpha+2}}
d\eta, ~\xi\in\R^n_+.
\end{eqnarray*}
Moreover,
\begin{eqnarray}\label{estequ1-UHST-mms-U-cri}
&&u_{y^0, \lambda} (\eta)-u(\eta)=\int_{\Sigma_{y^0,\lambda}^n}K(y^0,\lambda;\xi,\eta)(v_{y^0, \lambda}^\kappa (\xi)-v^\kappa (\xi))d\xi,\\\label{estequ2-UHST-mms-V-cri}
&&v_{y^0, \lambda} (\xi)-v(\xi)=\int_{\Sigma_{y^0,\lambda}^{n-1}}K(y^0,\lambda;\xi,\eta)(u_{y^0, \lambda}^\theta (\eta)-u^\theta (\eta))d\eta,
 \end{eqnarray}
where
\begin{equation*}
K(y^0,\lambda;\xi,\eta):=\frac{\xi_n}{|\xi-\eta|^{n-\alpha+2}}-\frac{\xi_n}{|\xi^{y^0,\lambda}-\eta|^{n-\alpha+2}}\cdot\big(\frac{\lambda}{|\xi-y^0|}\big)^{n-\alpha+2}>0,
\end{equation*}
for $\xi\in\Sigma_{y^0,\lambda}^n,\eta\in\Sigma_{y^0,\lambda}^{n-1},\lambda>0$.
\end{lm}
\begin{proof}
Direct computations yields the result. See for example, proof of  Lemma 3.2 in \cite{DZ-2}.
\end{proof}

\begin{lm}\label{lm-UHST-mms-start-cri}
Let $(u,v)$ be a pair of  positive solutions to system \eqref{equ-UHST-cri}  with $u(\eta)\in L^{\theta+1}(\partial \R^n_+),v(\xi)\in L^{t^\prime}(\R^n_+)$, where $\frac{3n}{2n-1}<\alpha<n$, $p=\frac{2(n-1)}{n+\alpha-4}, t=\frac{2n}{n+\alpha}$.  Then, for any $y^0\in\partial\R^n_+$, there exists $\lambda_0(y^0)>0$ small enough such that: $\forall~ 0<\lambda<\lambda_0(y^0)$,
\begin{eqnarray*}
u_{y^0, \lambda} (\eta)&\le & u(\eta),~\forall ~\eta\in\Sigma_{y^0,\lambda}^{n-1},\\
v_{y^0, \lambda} (\xi)&\le & v (\xi),~\forall ~\xi\in\Sigma_{y^0,\lambda}^n.
\end{eqnarray*}
\end{lm}

\begin{proof}
Let $$\Sigma_{y^0,\lambda}^{n,v}=\{\xi\in\Sigma_{y^0,\lambda}^n\ : \ v(\xi)<v_{y^0, \lambda} (\xi)\}, \ \mbox{and} \ \ \Sigma_{y^0,\lambda}^{n-1,u}=\{\eta\in\Sigma_{y^0,\lambda}^n\ : \ u(\eta)<u_{y^0, \lambda} (\eta)\}.$$
By Lemma \ref{lm-UHST-mms-UV-cri}, if $\eta\in \Sigma_{y^0,\lambda}^{n-1,u}$, we have
\begin{equation}\label{equ1-UHST-mms-start-cri}
0\le u_{y^0, \lambda} (\eta)-u(\eta)\le \int_{\Sigma_{y^0,\lambda}^{n,v}}\frac{\xi_n}{|\xi-\eta|^{n-\alpha+2}}(v_{y^0, \lambda}^\kappa (\xi)-v^\kappa (\xi))d\xi.
\end{equation}
Notice that $\kappa=t^\prime-1=\frac{n+\alpha}{n-\alpha}>1$.
By \eqref{equ1-UHST-mms-start-cri}, for $\eta\in \Sigma_{y^0,\lambda}^{n-1,u}$ we have
\begin{eqnarray*}
0&\le&  u_{y^0, \lambda} (\eta)-u(\eta)\\
&\le&\kappa  \int_{\Sigma_{y^0,\lambda}^{n,v}}\frac{\xi_n\Phi^{\kappa-1}_{y^0, \lambda}(v)(v_{y^0, \lambda}(\xi)-v(\xi))}{|\xi-\eta|^{n-\alpha+2}}d\xi\\
&\le&\kappa  \int_{\Sigma_{y^0,\lambda}^{n,v}}\frac{\xi_n v_{y^0, \lambda}^{\kappa-1}(\xi)(v_{y^0, \lambda}(\xi)-v(\xi))}{|\xi-\eta|^{n-\alpha+2}}d\xi,
\end{eqnarray*}
where $v(\xi)\le\Phi_{y^0, \lambda}(v)\le v_{y^0, \lambda}(\xi)$. By \eqref{ineq-HLST-UHS}, we have
\begin{eqnarray}\nonumber
 \|u_{y^0, \lambda} -u\|_{L^{p^\prime}(\Sigma_{y^0,\lambda}^{n-1,u})} 
&\le&C\|v_{y^0, \lambda}^{\kappa-1}(v_{y^0, \lambda}-v)\|_{L^t(\Sigma_{y^0,\lambda}^{n,v})}\\\nonumber
&\le&C\|v_{y^0, \lambda}\|^{\kappa-1}_{L^{t^\prime}(\Sigma_{y^o,\lambda}^{n,v})}
\|v_{y^0, \lambda}-v\|_{L^{t^\prime}(\Sigma_{y^0,\lambda}^{n,v})}\\\label{equ2-UHST-mms-start-cri}
&=&C\|v\|^{\kappa-1}_{L^{t^\prime}(B_\lambda^+(y^0))}
\|v_{y^0, \lambda}-v\|_{L^{t^\prime}(\Sigma_{y^0,\lambda}^{n,v})}.
\end{eqnarray}

Now we consider two cases.

Case (\rmnum1). $\theta>1$.

 Similar to the above, we have
\begin{eqnarray*}
\|v_{y^0, \lambda}-v\|_{L^{t^\prime}(\Sigma_{y^0,\lambda}^{n,v})}
\le C\|u\|^{\theta-1}_{L^{\theta+1}(B_\lambda^{n-1}(y^0))}\|u_{y^0, \lambda} -u\|_{L^{p^\prime}(\Sigma_{y^0,\lambda}^{n-1,u})}.
\end{eqnarray*}
Notice that  $u(\eta)\in L^{\theta+1}(\partial \R^n_+),v(\xi)\in L^{t^\prime}(\R^n_+)$. Thus  for $\lambda>0$ small enough, $\|v\|^{\kappa-1}_{L^{t^\prime}(B_\lambda^+(y^0))}, \|u\|^{\theta-1}_{L^{\theta+1}(B_\lambda^{n-1}(y^0))}$ will be close to zero, which yields  that
 $$\|u_{y^o, \lambda} -u\|_{L^{p^\prime}(\Sigma_{y^0,\lambda}^{n-1,u})}=\|v_{y^0, \lambda}-v\|_{L^{t^\prime}(\Sigma_{y^0,\lambda}^{n,v})}=0.$$
  That is: both $\Sigma_{y^0,\lambda}^{n,v}$ and $\Sigma_{y^0,\lambda}^{n-1,u}$ are empty.

Case (\rmnum2). $0<\theta\le 1$.

 Since $\frac1p+\frac1t>1,$ we have $1\le\frac1\theta<\kappa$. Thus there exists a number $r$ satisfied $\frac1\theta<r<\kappa $.
By \eqref{ineq-HLST-UHS-equv-C} and H\"{o}lder inequality, we have
\begin{eqnarray}\label{equ2-UHST-mms-start-cri-2}
\|v_{y^0, \lambda}-v\|_{L^{t^\prime}(\Sigma_{y^0,\lambda}^{n,v})}
\le C\|u\|^{\theta-\frac1r}_{L^{\theta+1}(B_\lambda^{n-1}(y^0))}\|u^\frac1r_{y^0, \lambda} -u^\frac1r\|_{L^{p^\prime}(\Sigma_{y^0,\lambda}^{n-1,u})}.
\end{eqnarray}
Similar to \eqref{equ2-UHST-mmp-start-2}, we can also conclude that
\begin{eqnarray*}
\|u^\frac1r_{y^0, \lambda} -u^\frac1r\|_{L^{p^\prime}(\Sigma_{y^0,\lambda}^{n-1,u})}
&\le&C\|v\|^{\frac{\kappa-r}r}_{L^{t^\prime}(B_\lambda^+(y^0))}
\|v_{y^0, \lambda}-v\|_{L^{t^\prime}(\Sigma_{y^0,\lambda}^{n,v})}.
\end{eqnarray*}
Combining the above with \eqref{equ2-UHST-mms-start-cri-2}, we obtain
\begin{equation*}
\|v_{y^0, \lambda}-v\|_{L^{t^\prime}(\Sigma_{y^0,\lambda}^{n,v})}\le C\|u\|^{\theta-\frac1r}_{L^{\theta+1}(B_\lambda^{n-1}(y^0))}\|v\|^{\frac{\kappa-r}r}_{L^{t^\prime}(B_\lambda^+(y^0))}
\|v_{y^0, \lambda}-v\|_{L^{t^\prime}(\Sigma_{y^0,\lambda}^{n,v})}.
\end{equation*}
Notice that  $u(\eta)\in L^{\theta+1}(\partial \R^n_+),v(\xi)\in L^{t^\prime}(\R^n_+)$. Thus  for $\lambda>0$ small enough, $\|u\|^{\theta-\frac1r}_{L^{\theta+1}(B_\lambda^{n-1}(y^0))}\|v\|^{\frac{\kappa-r}r}_{L^{t^\prime}(B_\lambda^+(y^0))}$ will be close to zero.
 Similar argument as that in Case (\rmnum1) concludes that both $\Sigma_{y^0,\lambda}^{n,v}$ and $\Sigma_{y^0,\lambda}^{n-1,u}$ are empty.
\end{proof}

Define $$\overline{\lambda}(y^0)=\sup\{\mu>0\ : \forall \lambda\in(0, \mu), \  \ u_{y^0, \lambda} (\eta)\le u(\eta),~\forall ~\eta\in\Sigma_{y^0,\lambda}^{n-1}; \ v_{y^0, \lambda} (\xi) \le  v (\xi),~\forall ~\xi\in\Sigma_{y^0,\lambda}^n\}.$$

\begin{lm}\label{lm-UHST-mms-infy-cri}
For some $y^0\in\partial R^n_+$, if $\overline{\lambda}(y^0)<\infty$, then
\begin{eqnarray*}
u_{y^0, \overline{\lambda}(y^0)} (\eta)&\equiv& u(\eta),~\forall ~\eta\in \partial \R^n_+,\\
v_{y^0, \overline{\lambda}(y^0)} (\xi)&\equiv& v (\xi),~\forall ~\xi\in \R^n_+.
\end{eqnarray*}
\end{lm}
\begin{proof}
For simplicity, we denote $\overline{\lambda}:=\overline{\lambda}(y^0)$. It is enough to prove that
\begin{eqnarray*}
u_{y^0, \overline{\lambda}} (\eta)&\equiv& u(\eta),~\forall ~\eta\in\Sigma_{y^0,\overline{\lambda}}^{n-1},\\
v_{y^0, \overline{\lambda}} (\xi)&\equiv& v (\xi),~\forall ~\xi\in\Sigma_{y^0,\overline{\lambda}}^n.
\end{eqnarray*}

Notice that
\begin{eqnarray*}
u_{y^0, \overline{\lambda}} (\eta)&\le & u(\eta),~\forall ~\eta\in\Sigma_{y^0,\overline{\lambda}}^{n-1},\\
v_{y^0, \overline{\lambda}} (\xi)&\le & v (\xi),~\forall ~\xi\in\Sigma_{y^0,\overline{\lambda}}^n.
\end{eqnarray*}
Then if $u_{y^0, \overline{\lambda}} (\eta)\not\equiv u(\eta)$ or $v_{y^0, \overline{\lambda}} (\xi)\not\equiv v (\xi)$, by Lemma \ref{lm-UHST-mms-UV-cri} we have
\begin{eqnarray*}
u_{y^0, \overline{\lambda}} (\eta)&< & u(\eta),~\forall ~\eta\in\Sigma_{y^0,\overline{\lambda}}^{n-1},\\
v_{y^0, \overline{\lambda}} (\xi)&< & v (\xi),~\forall ~\xi\in\Sigma_{y^0,\overline{\lambda}}^n.
\end{eqnarray*}
Notice that $u,v$ are continuous. Then for $R>0$ large we can take $\delta>0$ small such that
\begin{eqnarray*}
&&u(\eta)-u_{y^0, \overline{\lambda}} (\eta)\ge C_0,~\forall ~\eta\in\Sigma_{y^0,\overline{\lambda}+\delta}^{n-1}\cap B^{n-1}_R(y^0),\\
&&v(\xi)-v_{y^0, \overline{\lambda}} (\xi)\ge C_0,~\forall ~\xi\in\Sigma_{y^0,\overline{\lambda}+\delta}^n\cap B^+_R(y^0)
\end{eqnarray*}
for some $C_0>0$ small. Then we can choose $0<\epsilon<\delta$ small such that for  $\lambda\in(\overline{\lambda},\overline{\lambda}+\epsilon)$,
\begin{eqnarray*}
&&u(\eta)-u_{y^0, \lambda} (\eta)\ge 0,~\forall ~\eta\in\Sigma_{y^0,\overline{\lambda}+\delta}^{n-1}\cap B^{n-1}_R(y^0),\\
&&v(\xi)-v_{y^0, \lambda} (\xi)\ge 0,~\forall ~\xi\in\Sigma_{y^0,\overline{\lambda}+\delta}^n\cap B^+_R(y^0).
\end{eqnarray*}
And so
 $$\Sigma_{y^0,\lambda}^{n-1,u}\subset (\partial \R^n_+ \setminus B^{n-1}_R(y^0))\cup(\Sigma_{y^0,\lambda}^{n-1}\setminus \Sigma_{y^0,\overline{\lambda}+\delta}^{n-1}):=\Omega_{\lambda,R}^{n-1}$$
 and
 $$\Sigma_{y^0,\lambda}^{n,v}\subset (\R^n_+ \setminus B^+_R(y^0)) \cup (\Sigma_{y^0,\lambda}^n\setminus \Sigma_{y^0,\overline{\lambda}+\delta}^n):=\Omega_{\lambda,R}^n.$$
Denote $(\Omega_{\lambda,R}^{n-1})^*, (\Omega_{\lambda,R}^n)^*$ as the reflection of $\Omega_{\lambda,R}^{n-1}$ and $\Omega_{\lambda,R}^n$ respectively under the Kelvin transformation with respect to $\partial B_R(y^0)$. Then we have
$$(\Omega_{\lambda,R}^{n-1})^*=B^{n-1}_{\epsilon_1}(y^0)\cup(B^{n-1}_\lambda(y^0)\setminus B^{n-1}_{\lambda^2\setminus(\overline{\lambda}+\delta)}(y^0))$$
 and
 $$(\Omega_{\lambda,R}^n)^*=B^+_{\epsilon_1}(y^0)\cup(B^+_\lambda(y^0)\setminus B^+_{\lambda^2\setminus(\overline{\lambda}+\delta)}(y^0)),$$
where $\epsilon_1:=\frac{\lambda^2}{R}\to 0$ as $R\to\infty$.

Therefore we can first choose $R>0$ large, then $\delta>0$ small, then $0<\epsilon<\delta$ small such that for $\lambda\in(\overline{\lambda},\overline{\lambda}+\epsilon)$, both $\|v\|_{L^{t^\prime}(\Omega_{\lambda,R}^n)^*}$ and $\|u\|_{L^{\theta+1}(\Omega_{\lambda,R}^{n-1})^*}$ are small enough. Then the same argument as Lemma \ref{lm-UHST-mms-start-cri} yields that $\Sigma_{y^0,\lambda}^{n,v}$ and $\Sigma_{y^0,\lambda}^{n-1,u}$ are empty, which contradicts with the definiton of $\overline{\lambda}$.
\end{proof}

{\bf Proof of Proposition \ref{prop-extr-sys-HLST}.}
If  $\overline{\lambda}(y)=\infty$ for any $y\in\partial \R^n_+$, then for any $\lambda>0$,
$$u_{y, \lambda}(\eta)\le u(\eta),~\forall \eta\in\Sigma_{y,\lambda}^{n-1}.$$
Lemma 5.7 in \cite{Li} implies that $u \equiv C_0>0$ is a constant, which contradicts with $u(\eta)\in L^{\theta+1}(\partial \R^n_+)$.
Thus there exists $y^0\in\partial \R^n_+$ such that $\overline{\lambda}(y^0)<\infty$. Now we prove that for any $y\in\partial \R^n_+$ it holds $\overline{\lambda}(y)<\infty$.

In fact, for any fixed $y\in\partial \R^n_+$, we know that for any $\lambda\in(0,\overline{\lambda}(y))$,
$u_{y, \lambda}(\eta)\le u(\eta),~\forall \eta\in\Sigma_{y,\lambda}^{n-1}$, which implies that for any
$\lambda\in(0,\overline{\lambda}(y))$,
\begin{equation}\label{equ1-prop-extr-sys-HLST}
a:=\liminf\limits_{|\eta|\to\infty}(|\eta|^{n-\alpha+2}u(\eta))\ge\liminf\limits_{|\eta|\to\infty}(|\eta|^{n-\alpha+2}u_{y, \lambda}(\eta))=|\lambda|^{n-\alpha+2}u(y).
\end{equation}
On the other hand, since $\overline{\lambda}(y^0)<\infty$, by using Lemma \ref{lm-UHST-mms-infy-cri},
\begin{equation}\label{equ2-prop-extr-sys-HLST}
a=\liminf\limits_{|\eta|\to\infty}(|\eta|^{n-\alpha+2}u(\eta))=\liminf\limits_{|\eta|\to\infty}(|\eta|^{n-\alpha+2}u_{y^0, \overline{\lambda}(y^0)}(\eta))=|\overline{\lambda}(y^0)|^{n-\alpha+2}u(y^0)<\infty.
\end{equation}
By using \eqref{equ1-prop-extr-sys-HLST} and \eqref{equ2-prop-extr-sys-HLST} we conclude that for any $y\in\partial \R^n_+$ it holds $\overline{\lambda}(y)<\infty$.
It then follows from  Lemma \ref{lm-UHST-mms-infy-cri}  that $$u_{y, \overline{\lambda}(y)}(\eta)\equiv u(\eta),~\forall y, \eta\in \partial \R^n_+.$$
 Lemma 5.8 in \cite{Li} yields that for any $\eta\in \partial \R^n_+$
$$u(\eta)=c_1\big(\frac{1}{|\eta-\eta^0|^2+d^2}\big)^{\frac{n-\alpha+2}{2}}$$
for some $c_1,d>0,\eta^0\in\partial \R^n_+.$
Similarly, we have for any $\eta\in \partial \R^n_+$
$$v(\eta,0)=c_2\big(\frac{1}{|\eta-\eta^0|^2+d^2}\big)^{\frac{n-\alpha}{2}}$$
for some $c_2,d>0,\eta^0\in\partial \R^n_+.$
\hfill$\Box$



\end{document}